\documentclass[a4paper,twoside]{article}
\usepackage{geometry}
\usepackage{fancyhdr}
\usepackage{amsmath}
\usepackage{amssymb}
\usepackage{amsthm}
\usepackage{enumerate}

\newtheorem{thm}{Theorem}[section]
\newtheorem{lem}{Lemma}[section]

\newtheorem{prop}{Proposition}[section]
\newtheorem{rem}{Remark}[section]
\newtheorem{defi}{Definition}[section]

\newcommand{\eps}{\varepsilon}
\newcommand{\R}{\mathbb{R}}

\newcommand{\N}{\mathbb{N}}

\renewcommand{\div}{{\rm div}\,}
\newcommand{\Id}{{\rm Id}\,}

\newcommand{\loc}{{\rm loc}\,}

\newcommand{\Int}{\displaystyle \int}

\def\ud{\underline}
\def\ov{\overline}
\def\d{\partial}

\def\dj{\Delta_j}

\def\tilde{\widetilde}
\def\hat{\widehat}
\def\div{{\rm div}\,}

\def\cC{{\mathcal C}}
\def\cD{{\mathcal D}}

\def\cK{{\mathcal K}}

\def\cM{{\mathcal M}}
\def\cP{{\mathcal P}}

\def\cR{{\mathcal R}}
\def\cS{{\mathcal S}}

\begin{document}

\title{A global existence result for a zero Mach number system}
\author{Xian LIAO\thanks{Universit\'{e} Paris-Est Cr\'{e}teil, LAMA, UMR 8050, 61 Avenue de G\'{e}n\'{e}ral de Gaulle, 94010 Cr\'{e}teil Cedex, France. Email: liaoxian925@gmail.com}}

\date{16 November 2012}

\maketitle

\begin{abstract}
\noindent{\textbf {Abstract. }}This paper is to study  global-in-time existence of weak solutions to zero Mach number system which derives from the full Navier-Stokes  system, under a special relationship  between the  viscosity coefficient and the  heat conductivity coefficient such that, roughly speaking, the source term in the equation for the newly introduced divergence-free velocity vector field vanishes.
In dimension two, thanks to a local-in-time existence result of a unique strong solution in  critical Besov spaces  given in \cite{Danchin-Liao},   for arbitrary large initial data,  we will show that this unique strong solution exists globally in time, by a weak-strong uniqueness argument.
\\

\\

\end{abstract}


\section{Introduction}

\subsection{Derivation of a zero Mach number system}

The free evolution of a viscous and heat conducting compressible Newtonian fluid obeys the following  Navier-Stokes  system:
\begin{equation}\label{eq:fullnewtonian}
\left\{
\begin{array}{ccc}
\d_t\rho+\div(\rho v)&=&0,\\
\d_t(\rho v)+\div(\rho v\otimes v)-\div \sigma+\nabla p&=&0,\\
\d_t(\rho e)+\div(\rho v e)-\div(k \nabla \vartheta)+p\div v&=&\sigma\cdot Dv.
\end{array}
\right.
\end{equation}

The above System \eqref{eq:fullnewtonian} describes the evolution of the mass density $\rho=\rho(t,x)$, the momentum $\rho v=\rho(t,x)v(t,x)$ and   the internal energy $\rho e=\rho(t,x)e(t,x)$ respectively, where $t\in \R^+$ is the positive time variable and the space variable $x$ belongs to $\R^d$ with $d\geq 2$. The scalar functions $p=p(t,x)$ and $\vartheta=\vartheta(t,x)$ denote the pressure and the temperature respectively, and the viscous strain tensor $\sigma$ is given by
$$\sigma:=2\mu  Sv+\nu \, \div v \,\Id,$$
where  $\mu=\mu(\rho,\vartheta)$ and $\nu=\nu(\rho,\vartheta)$ are the regular viscosity coefficients, $\Id$ is the  identity tensor and the symmetric rate-of-deformation tensor is denoted by
$$
Sv:=\frac 12(\nabla v+Dv)\quad\textrm{with}\quad (\nabla v)_{ij}:=\d_i v^j,\quad (Dv)_{ij}:=\d_j v^i.
$$
The thermal conductivity coefficient  $k=k(\rho,\vartheta)$ is also a smooth function of $\rho$ and $\vartheta$.

There should be another two equations of  state   such that System \eqref{eq:fullnewtonian} is complete and  in this work, we will take  \emph{ideal gas} model:
\begin{equation}\label{state}
p=R\rho\vartheta,\qquad e=C_v\vartheta,
\end{equation}
where the two constants $R$ and $C_v$, denote the ideal gas constant and the specific heat capacity at constant volume respectively.


In this paper we will consider  a  zero Mach number  system, which derives  from System \eqref{eq:fullnewtonian} by letting Mach number vanish and admitting a specific relationship between variable physical coefficients $k$ and $\mu$  (see \eqref{cond:coeff} below). That is, we assume the fluid to be highly subsonic and hence the compression due to pressure is neglectable and in addition, we restrict ourselves to effectively heat-conducting and viscous fluids.

More precisely, firstly, we introduce the dimensionless Mach number $\eps$ as the ratio of the velocity $v$ over the reference sound speed, then  the rescaled triplet
$$
\left(\rho_\eps(t,x)=\rho\Bigl(\frac{t}{\eps},x\Bigr),\quad
v_\eps(t,x)=\frac{1}{\eps}v\Bigl(\frac{t}{\eps},x\Bigr),\quad
\vartheta_\eps(t,x)=\vartheta\Bigl(\frac{t}{\eps},x\Bigr)\right)
$$
satisfies
\begin{equation}\label{eq:epsilon}\left\{
\begin{array}{ccc}
\d_t\rho_\eps+\div(\rho_\eps v_\eps)&=&0,\\[1ex]
\d_t(\rho_\eps v_\eps)+\div (\rho_\eps v_\eps\otimes v_\eps)-\div\sigma_\eps+\frac{\nabla p_\eps}{\eps^2}&=&0,\\[1ex]
\frac{1}{\gamma-1}(\d_tp_\eps+\div(p_\eps v_\eps ))-\div(k_\eps\nabla \vartheta_\eps)+p_\eps\div v_\eps&=&\eps^2\sigma_\eps\cdot Dv_\eps,
\end{array}
\right.
\end{equation}
where $\gamma:=C_p/{C_v}=1+R/{C_v}$ with the constant $C_p$ denoting the specific heat capacity at constant pressure and
$$
\sigma_\eps=2\mu_\eps Sv_\eps+\nu_\eps \div v_\eps \Id,\quad p_\eps=R\rho_\eps\vartheta_\eps,\quad \varsigma_\eps(t,x)=\frac{1}{\eps}\varsigma\Bigl(\frac{t}{\eps},x\Bigr)\hbox{ with }\varsigma=\mu,\nu,k\hbox{ respectively}.
$$

Let $\eps$ go to $0$, that is, the pressure $p_\eps$ equals to a positive constant $P_0$ by Equations $\eqref{eq:epsilon}_2$ and $\eqref{eq:epsilon}_3$,  and hence the   rescaled system \eqref{eq:epsilon} becomes the following zero Mach number system immediately (see \cite{Alazard06}, \cite{PLions96} or \cite{Zeytounian04} for detailed computation):
\begin{equation}\label{system_origin}
\left\{ \begin{array}{ccc}
\d_t \rho+\div(\rho v) & =&0,\\
 \d_{t}(\rho v)+\div(\rho v\otimes v )-\div\sigma + \nabla\Pi & = & 0, \\
 \div v - \div(\alpha k \nabla\vartheta)& = &0,
\end{array} \right.
\end{equation}
with   $\alpha:=\frac{\gamma -1}{\gamma P_0}$ being a positive constant.
Let us point out here that a rigorous justification of the above low-Mach number limit System \eqref{system_origin} has been presented by T. Alazard, see  \cite{Alazard06}.

The unknowns in the above system \eqref{system_origin} turn to be the density $\rho$, the velocity field $v$ and some unknown ``pressure''  $\Pi$.
Thanks to equations of state \eqref{state}, the temperature $\vartheta$ is thus equal to $P_0/(R\rho)$.
 Consequently, all the physical coefficients $k,\mu,\nu$ can be viewed as  regular functions of the density $\rho$ only.
 Set  $\kappa=\kappa(\rho)=\alpha k\vartheta$, then Equation $\eqref{system_origin}_3$ becomes
 \begin{equation}\label{eq:div}
 \div(v+\kappa\nabla\ln\rho)=0.
 \end{equation}


 In order to take advantage of the ``almost divergence-free'' condition of Equation $\eqref{system_origin}_3$, just as in \cite{Danchin-Liao},  we introduce a new  solenoidal ``velocity'' field $u$ as follows:
\begin{equation*}\label{transform}
u=v-\alpha k\nabla\vartheta=v+\kappa\nabla\ln\rho.
\end{equation*}
We  next try to rewrite the terms concerning the original velocity field $v$ in System \eqref{system_origin} in light of the newly introduced velocity $u$.
Firstly, it is easy to see that
\begin{equation*}\label{u,v}
\rho v =\rho u-\kappa\nabla\rho.
\end{equation*}
It is hence easy to write $\d_t(\rho v)$ as (noticing that $\kappa$ can be viewed as a regular function of $\rho$ only):
\begin{equation*}\label{transform:d_t}
\d_t(\rho v)=\d_t (\rho u)-\nabla(\kappa\d_t\rho).
\end{equation*}
We can also rewrite the convection term in the conservation law of momentum  as following:
\begin{equation*}
\div (\rho v\otimes v)=\div(\rho v\otimes u)+\div(v\otimes (-\kappa\nabla\rho)).
\end{equation*}
Observing the following two equalities
$$
Sv-Dv=Au:=\frac12 (\nabla u-Du)\,
\quad\textrm{and}\quad-\div(\nu\, \div v\,\Id)=-\nabla(\nu\,\div v)
,$$
the viscosity term $-\div\sigma$ thus reads as
\begin{equation*}
-\div(2\mu Au)-\div(2\mu Dv)-\nabla(\nu\div v)
=-\div(2\mu Au)+\div(v\otimes 2\nabla\mu)-\nabla(\div(2\mu v)+\nu\div v).
\end{equation*}

Therefore, if we admit the following relationship between the heat conductivity coefficient $k$ and the viscosity coefficient $\mu$:
\begin{equation}\label{cond:coeff}
-\kappa(\rho)+2\mu'(\rho)=0,\quad\textrm{i.e.}\quad k(\vartheta)+2C_p\vartheta \mu'(\vartheta)=0,
\end{equation}
then by  introducing a  new ``pressure'' $\pi$ as
$$\pi=\Pi-\kappa\d_t\rho-\div(2\mu\, v)-\nu\,\div v,$$
 System (\ref{system_origin}) recasts in the following system for the new unknown triplet $(\rho, u,\nabla\pi)$:
\begin{equation}\label{system}\left\{\begin{array}{ccc}
\d_{t}\rho+\div(\rho v)=\d_t \rho+\div(\rho u) -\div(\kappa\nabla\rho) & = & 0,\\
\d_{t}(\rho u)+\div(\rho v \otimes u)-\div(2\mu A u)+\nabla\pi&=&0,\\
\div u&=&0.
\end{array}\right.\end{equation}

It is easy to check that if the triplet $(\rho,u,\nabla\pi)$ solves the above system \eqref{system} in distribution sense  which, also satisfies
\begin{align}\label{cond:reverse}
&\rho(t,x)\in [\ud\rho,\ov\rho],\,\forall t\in [0,+\infty),\,x\in \R^d,\nonumber\\
&\rho-c_0\in L^\infty\Bigl([0,+\infty);L^2(\R^d) \Bigr)\cap L^2_\loc\Bigl([0,+\infty);H^1(\R^d)\Bigr),\\
& \d_t\rho\in L^2_\loc\Bigl([0,+\infty);H^{-1}(\R^d)\Bigr),\nonumber\\
& u,\, v:=u-\kappa\nabla\ln\rho \in L^2_\loc\Bigl([0,+\infty);(L^2(\R^d))^d\Bigr),\nonumber
\end{align}
with three positive constants $\ud\rho,\ov\rho,c_0$, then the above calculation  from System \eqref{system_origin} to System \eqref{system} makes sense and $(\rho,v)$, together with some distribution  pressure $\Pi$, satisfies System \eqref{system_origin} at least in the sense of distribution, under the coefficient condition \eqref{cond:coeff}.

 Notice that System \eqref{system_origin} also can be viewed as Kazhikhov-Smagulov type models, which describe the motion of mixtures such as a salt or pollutant spread  in a compressible fluid, see \cite{BV}.
  There, the divergence of the velocity $v$ is related to the derivatives of the density $\rho$ by Fick's law
  \begin{equation}\label{Fick}
  \div v+\div(\kappa_0\nabla\ln \rho)=0,
  \end{equation}
  where the diffusive coefficient $\kappa_0$ is a positive  constant.
  Therefore, in that case, by view of \eqref{eq:div}, Relation \eqref{cond:coeff} implies that the viscosity coefficient $\mu$ equals to $\kappa_0\rho/2$, up to a constant.
  Another interesting example  proposed in \cite{Majda} is a low Mach number combustion model  with a constant thermic coefficient $k$, where   Relation \eqref{cond:coeff} entails  $\mu=-k\ln\vartheta/(2C_p)$, up to a constant.
More generally, by virtue of \eqref{cond:coeff}, we consider the gases which become less viscous as the temperature increases (or the density decreases) and moreover, if there is more thermal conduction, the viscosity decreases faster.
However, unlike the physical coefficients $k$ and $\mu$, the other viscosity coefficient $\nu$ plays no role here.

\subsection{Main results}
In this work, we want to consider Cauchy problem of System \eqref{system_origin} globally in time.
There is  a long history of   global existence problem of  weak  solutions (roughly speaking, solutions in the sense of distribution with bounded physical energy) with \emph{large} initial data  to Navier-Stokes system.
As early as in 1934, J. Leray in \cite{Leray34} proved such a global existence result to the classical incompressible Navier-Stokes equations (System \eqref{system_origin} with constant density and temperature) in dimension $d=2$ and $3$, and especially in dimension $2$ global regularity and uniqueness also hold.
 Another big breakthrough is due to P.-L. Lions, who treated a somehow simplified compressible Navier-Stokes system (the first two equations in System \eqref{eq:fullnewtonian} under the gamma-type pressure law  $p(\rho)=a\rho^\gamma$, $a>0$)  in the nineties of last century.
  Indeed in \cite{PLions96vol2}, observing the so-called effective viscous flux, he showed that if $\gamma\geq 9/5$ in dimension $3$, weak solutions exist globally in time.
 In the very beginning of this century, E. Feireisl  improved this result to the case $\gamma>3/2$, see \cite{Feireisl}.
  However, for the compressible case, global regularity or uniqueness for large data is still unknown, even in dimension $d=2$,  due to a lack of estimation in the vacuum region.
 Recently in  \cite{B-D}, D. Bresch and B. Desjardins  extended this type of result to   the full Navier-Stokes system \eqref{eq:fullnewtonian}   by using the so-called BD-entropy,  under some specific assumptions on   equations of state and physical coefficients which are different from the coefficient relationship \eqref{cond:coeff} here.
  Let us  emphasize also that, in \cite{Feireisl04}, E. Feireisl presented a global existence result for System \eqref{eq:fullnewtonian}, but referring to the so-called ``variational'' solutions.

On the other side, there are also numerous works involving  global-in-time  existence of strong  solutions  (unique and regular in general) to Navier-Stokes system, but with \emph{small} initial data in general.
Let's just mention that for the classical incompressible Navier-Stokes equations, in 1964, H. Fujita and T. Kato  obtained a unique global  solution, see \cite{Fujita-Kato}.
With high regularity assumptions, in dimension $3$, A. Matsumura and T. Nishida \cite{M-N} studied   the motion of  viscous and heat-conductive gases.
In \cite{Danchin00} and \cite{Danchin01global}, R. Danchin considered   the movement of  barotropic compressible fluids and   compressible viscous and heat-conducting gases respectively, but in critical Besov spaces.

When there is no heat conduction, System \eqref{system_origin} becomes the density-dependent incompressible Navier-Stokes system.
See the book \cite{A-K-M} and the references therein for the existence results of the associated initial-boundary value problem.
Global well-posedness of the Cauchy problem   was demonstrated by R. Danchin in \cite{Danchin03} and  in addition in dimension $2$, he also examined \emph{large} initial data   case  in \cite{Danchin04}.
There are also a few works devoted to  global-in-time solutions to the Cauchy problem of zero Mach number  system \eqref{system_origin}, under some smallness assumptions on the diffusion coefficient $\kappa$ or on the initial data in general.
In the pioneering work \cite{Kazhikhov-Smagulov}, A.V. Kazhikhov and Sh. Smagulov addressed the initial boundary value problem of  a somehow \emph{simplified} mathematical model for a two-component mixture involving mass diffusion inside.
More precisely, the constant diffusion coefficient $\kappa_0$ in the Fick law \eqref{Fick} was assumed to be \emph{small} compared with the \emph{constant} viscosity coefficient $\mu$.
 Besides, after taking the transformation \eqref{transform}, they just neglected the term $\div(\kappa_0\nabla\ln\rho\otimes(-\kappa_0\nabla\rho))$ in the convection term in Equation $\eqref{system_origin}_2$, which is of order $\kappa_0^2$.
 Finally,  for the system \eqref{system} with an additional convection term $-\kappa_0 u\cdot\nabla^2\rho$ on the left side, they obtained a global-in-time generalized solution for finite-energy initial data with the initial inhomogeneity $\rho_0-1$ in $H^1$.
If in addition the initial velocity $u_0$ belongs to $H^1$, this global solution is unique in dimension two.
 Under a similar smallness hypothesis on $\kappa_0$, P. Secchi \cite{Secchi88} got a unique global classical solution in dimension two for System \eqref{system_origin}.
 There, he also studied the asymptotic behaviour when $\kappa_0$ goes to zero.
 The general case (i.e. $\kappa_0$ can be variable and arbitrarily large) was considered by Beir\~{a}o da Veiga \cite{BV, BV96}.
 See also \cite{B-S-V} for the inviscid case.
 Under exactly the  condition \eqref{cond:coeff} on the physical coefficients, \cite{Sy02, Sy07} gave the existence of global-in-time weak solutions in smooth bounded domains.
Furthermore in dimension two, a recent work \cite{CLS} proved that the weak solutions are in fact unique, under the same initial condition as in Kazhikhov-Smagulov \cite{Kazhikhov-Smagulov} above.
Let us also mention here that for a low Mach number combustion model proposed in \cite{Majda} which, in addition to System \eqref{system_origin}, also takes into account a reaction-diffusion equation for the different species,  P. Embid \cite{Embid87} got  a unique \emph{local-in-time} regular solution.
In Section 8.8 in \cite{PLions96vol2}, P.-L. Lions indicated that in dimension $2$, \emph{small} and smooth enough perturbation will indeed evolve a global weak solution to System \eqref{system_origin}, if the heat conductivity coefficient $k$ is a positive constant.
 Recently in \cite{Danchin-Liao}, R. Danchin and the author studied System \eqref{system_origin} where \emph{variable} positive physical coefficients case was considered.
   There, without assuming any specific coefficient relationship,  a global strong solution was obtained, for  \emph{small} initial perturbation in any dimension $d\geq 2$.

At last, let us just mention some results on the low Mach number limit.
After the pioneer works by D.G. Ebin \cite{Ebin77, Ebin82} and S. Klainerman and A. Majda \cite{Klainerman-M82}, the incompressible limit   has attracted many mathematicians' attention.
Among them, we cite \cite{ Alazard05, Isozaki87, Isozaki89, Metivier-S01, Schochet86, Ukai} for the inviscid case while \cite{Alazard06, BV94,  Danchin02limit, Desjardins-G99, Desjardins-G-L-M99, FN07, Hoff98,  PLions-M98} for Navier-Stokes equations.
See also \cite{BV09} for a good summary.

The aim of this paper is to prove the global-in-time existence of  weak  solutions of Leray's type to  zero Mach number system \eqref{system_origin}   where the two positive variable physical coefficients $k,\mu$ satisfy  Condition \eqref{cond:coeff}, on the whole space $\R^d$, $d\geq 2$.
Furthermore, thanks to the local-in-time existence of a unique strong solution to System \eqref{system_origin} given in \cite{Danchin-Liao}, we can also show that in dimension $2$, for any initial data  of critical regularity,  the weak solutions are in fact   strong solutions  and hence   unique and regular.
This is, to our knowledge, the first result of this kind for  System  \eqref{system_origin}.

\smallbreak

Let's first   analyze System \eqref{system} \emph{formally}. If we assume the initial density $\rho_0(x)$ to be bounded from below and above by two positive constants $\ud\rho$ and  $\ov\rho$ respectively, then  the maximum principle for the  parabolic  equation $(\ref{system})_1$ ensures the density to satisfy the following uniform bound:
\begin{equation}\label{bound:density}
\rho(t,x)\in [\ud\rho,\ov\rho],\quad\forall\, t\geq 0,\, x\in \R^d.
\end{equation}
  Furthermore, if $\rho$ is close to a constant, say ``1'', at infinity, then  we expect the solution $(\rho,u)$  to satisfy the following two energy equalities which come from taking the $L^2(\R^d)$-inner product between $\rho-1$ (resp. $u$) and $(\ref{system})_1$ (resp. $(\ref{system})_2$):
\begin{equation}\label{energy iden:density}
\Int_{\R^d}|\rho(t)-1|^2+2\Int^t_0\Int_{\R^d}\kappa|\nabla\rho|^2
=\Int_{\R^d}|\rho_0-1|^2,\quad \forall\, t>0\,,
\end{equation}
and
\begin{equation}\label{energy iden:velocity}
\Int_{\R^d}\rho(t)|u(t)|^2+4\Int^t_0\Int_{\R^d}\mu|Au|^2=\Int_{\R^d}\rho_0|u_0|^2,\quad \forall\, t>0\,.
\end{equation}


Therefore, we  complement System (\ref{system}) with  initial data $(\rho_0,u_0)$ verifying
\begin{equation}\label{initial}
\rho-1|_{t=0}=\rho_0-1\in L^2(\R^d),\quad u|_{t=0}=u_0\in (L^2(\R^d))^d,\quad 0<\ud\rho\leq \rho_0\leq \ov \rho,\quad \div u_0=0,
\end{equation}
hoping that the obtained solution $(\rho,u)$ satisfies (\ref{energy iden:density}) and \eqref{energy iden:velocity}, at least in inequality form.

Moreover, we assume that under the bounds \eqref{bound:density} imposed on the density, there exist positive constants $\ud k,\ov k,\ud\mu,\ov\mu$ depending only on $\ud\rho,\ov\rho$ such that the physical coefficients  $k$ and  $\mu$ also have positive lower and upper bounds:
\begin{equation}\label{bdd:coeff}
0<\ud k\leq k(t,x)\leq \ov k,\quad 0<\ud\mu\leq \mu(t,x)\leq \ov\mu,\quad\forall\, t\geq 0,\,x\in \R^d.
\end{equation}


To conclude, we have the following global existence result for System \eqref{system}:
\begin{thm}\label{thm:global}
There exists a global-in-time   weak solution   $(\rho,u)$  to   Cauchy problem \eqref{system}-\eqref{initial} in the following sense:
\begin{itemize}
\item $\rho-1\in C([0,+\infty);L^p(\R^d))\cap L^2_\loc([0,+\infty);H^1(\R^d)),\,\forall\, p\in [2,\infty)$.\\
$\rho$ satisfies $\eqref{system}_1$  in $L^2_\loc([0,+\infty);H^{-1}(\R^d))$ and $\rho(0)=\rho_0$.\\
The uniform bound \eqref{bound:density} and \emph{Energy Equality} \eqref{energy iden:density} both hold.

\item For any $t>0$ and any test function $\phi(t,x)\in (C^\infty([0,+\infty)\times \R^d))^d$ with compact support such that $\div \phi=0$, the following holds:
\begin{equation}\label{weak solution:u}
\Int_{\R^d}\rho(t)u(t)\cdot\phi-\Int_{\R^d}\rho_0 u_0\cdot \phi|_{t=0}
-\Int^t_0\Int_{\R^d}\left[ \rho u\cdot\d_\tau \phi
+  (\rho u-\kappa\nabla\rho)\cdot\nabla\phi\cdot u-2\mu Au:A\phi\right]=0.
\end{equation}

\item $u\in C([0,\infty);(L^2_w(\R^d))^d)\cap L^2_\loc([0,+\infty);(H^1(\R^d))^d)$, $u(0)=u_0$ and $\div u=0$ in $L^2_w(\R^+\times \R^d)$ with $L^2_w$ denoting the Lebesgue space $L^2$ endowed with weak topology.\\
There exists a positive constant $C$ depending only on $\ud\rho,\ov\rho$ such that $u$ verifies the energy inequality
\begin{equation}\label{u,energy}
\Int_{\R^d}|u(t)|^2+\Int^t_0\Int_{\R^d}|\nabla u|^2\leq C\Int_{\R^d}|u_0|^2,\quad \forall t>0.
\end{equation}

\end{itemize}
\end{thm}


In the above, it is easy to check (by view of Condition \eqref{cond:reverse}) that  $\rho$ and $v:=u-\kappa\nabla\ln\rho\in L^2([0,+\infty);(L^2(\R^d))^d)$ satisfy System \eqref{system_origin}-\eqref{cond:coeff}  in distribution sense.
But  owing to a lack of high regularity assumption on the initial density $\rho_0$, we  don't know the continuity of $v$ at the initial instant, since we can't even define the quantity $\kappa(\rho_0)\nabla\ln\rho_0$.

Instead, if we assume  $\rho_0-1\in H^1(\R^d)$ in addition to  the initial condition \eqref{initial}, that is, the initial original velocity field $v_0$ belongs to $(L^2(\R^d))^d$ too,  then we expect that there exists  a global weak solution $(\rho-1,v)\in H^1(\R^d)\times (L^2(\R^d))^d$ to System \eqref{system_origin}-\eqref{cond:coeff}, if $d=2$ or $3$.
In fact, if $\kappa\equiv 1$, then taking $L^2(\R^d)$-inner product between Equation $\eqref{system}_1$ and $\Delta\rho$ yields
\begin{equation}\label{EE:rho-H1}
\frac12\frac{d}{dt}\|\nabla\rho\|_{ L^2(\R^d) }^2+\|\Delta\rho\|_{ L^2(\R^d) }^2
\leq \left |\Int_{\R^d}\div(\rho u)\Delta\rho \right|\,.
\end{equation}
Since a priori we have the  estimate (noticing $\div u=0$)
$$
\left|\Int_{\R^d}\div(\rho u)\Delta\rho \right|
=\left|\Int_{\R^d}\nabla\rho\cdot\nabla u\cdot\nabla\rho \right|
\leq \|\nabla u\|_{L^2(\R^d)}\|\nabla\rho\|_{L^4(\R^d)}^2,
$$
and the following two interpolation inequalities
\begin{equation}\label{ineq:rho-L4}
\|\nabla\rho\|_{L^4(\R^2)}^2\lesssim
       \|\Delta\rho\|_{L^2(\R^2)}\|\nabla\rho\|_{L^2(\R^2)},\quad
\|\nabla\rho\|_{L^4(\R^3)}^2\lesssim
       \|\Delta\rho\|_{L^2(\R^3)}\| \rho\|_{L^\infty(\R^3)},
\end{equation}
we have from  Young's Inequality and Estimate \eqref{u,energy} the following two energy estimates:
\begin{align}\label{EE:rho-H1,3D}
& \|\nabla\rho\|_{L^\infty_t( L^2(\R^2) )}+\|\Delta\rho\|_{L^2_t( L^2(\R^2)) }
\leq C\|\nabla\rho_0\|_{L^2}\,e^{C\int^t_0 \|\nabla u\|_{L^2}^2 }
\leq C \|\nabla\rho_0\|_{L^2(\R^2)}\,e^{C\|u_0\|_{L^2(\R^2)}^2}\,,\nonumber\\
& \|\nabla\rho\|_{L^\infty_t( L^2(\R^3) )}+\|\Delta\rho\|_{L^2_t( L^2(\R^3)) }
\leq  C(\|\nabla\rho_0\|_{L^2}+\|\nabla u\|_{L^2_t(L^2)})
\leq  C \|(\nabla\rho_0,u_0)\|_{L^2(\R^3)} \,,
\end{align}
for some constant $C$ depending also on $\ud\rho, \ov\rho.$
In general case where $\kappa$ depends on $\rho$,  we consider, instead,  the equation for the scalar function $K=K(\rho)$ with $\nabla K=\kappa\nabla\rho$ and $K(1)=0$, see Section \ref{sec:H1} in the following for more details.

But it is not clear that we can still have  $v \in L^\infty_t((L^2(\R^d))^d)$, in dimension $d\geq 4$, due to a lack of an interpolation inequality like \eqref{ineq:rho-L4} which can be  used to control   the convection term $u\cdot\nabla\rho$ in the equation of the density.
Anyway, we have
\begin{thm}\label{thm:global,v}
Let $d=2, 3$ and Relation \eqref{cond:coeff} hold.
 For any initial data $(\rho_0,v_0)$ such that
\begin{equation}\label{initial,v}
\rho_0-1\in H^1(\R^d),\quad v_0\in (L^2(\R^d))^d,
\quad 0<\ud\rho\leq \rho_0\leq \ov\rho,
\quad \div(v_0 +\kappa(\rho_0)\nabla\ln\rho_0)=0,
\end{equation}
there exists a global-in-time weak solution   $(\rho,v)$ to  Cauchy problem \eqref{system_origin}-\eqref{initial,v}  in the sense given in Theorem \ref{thm:global}, except with Equality \eqref{weak solution:u} replaced by
\begin{equation}\label{weak solution:v}
\Int_{\R^d}\rho(t)v(t)\cdot\phi-\Int_{\R^d}\rho_0 v_0\cdot \phi|_{t=0}
-\Int^t_0\Int_{\R^d}\left[ \rho v\cdot\d_\tau \phi
+   \rho v \cdot\nabla\phi\cdot v- \sigma:D\phi\right]=0.
\end{equation}
Furthermore, $(\rho,v)$ satisfies
 \begin{align}\label{sol:H1}
 &(\rho-1,v)\in C\Bigl([0,+\infty);H^1(\R^2)\times (L^2(\R^2))^2\Bigr),\nonumber\\
 &\rho-1\in C\Bigl([0,+\infty);H^s(\R^3)\Bigr),\,\forall\, s<1,\quad \hbox{and}\quad
 v\in C\Bigl([0,+\infty);(L^2_w(\R^3))^3\Bigl).
 \end{align}
and the following energy estimate:
\begin{equation}\label{rho,v,energy}
\|(\rho-1,v)\|_{L^\infty_t([0,\infty);H^1\times L^2 )}
+\|(\nabla\rho,\nabla v)\|_{ L^2_t([0,\infty);H^1\times  L^2 )}
\leq C(\ud\rho,\ov\rho,\|(\rho_0-1,v_0)\|_{H^1\times  L^2 }).
\end{equation}
\end{thm}


Next we want to show the global solutions got above are actually unique for some reasonably smooth initial data when $d=2$.
We notice that even if $\rho-1\in L^\infty(H^1)$, it is difficult to show the uniqueness.
 In fact, if we consider the system for the difference of any two solutions, then the nonlinear terms in System (\ref{system_origin}) ask the $L^\infty(L^\infty)$-norm control on the difference of  two densities.
 It is unknown a priori because the difference does not (at least not obviously) satisfy any parabolic equation  and another unlucky thing is that we can't embed  $H^1(\R^2)$ into $L^\infty(\R^2)$.
  Therefore, we have to resort to  Besov functional space $B^1_{2,1}(\R^2)$ (see Definition \ref{def:Besov} below)  which can be embedded both in $H^1$ and in $L^\infty$ in dimension $2$ and furthermore, we have already  the following existence result, as a special case of Theorem 1.2 in \cite{Danchin-Liao}:
\begin{thm}\label{prop:solution-besov}
In dimension $2$, for any  initial density $\rho_0$ and velocity field $u_0$ which satisfy
\begin{equation}\label{cond:2d,initial}\begin{array}{c}
0<\ud\rho\leq  \rho_0\leq \ov\rho,\quad\div u_0=0\quad
\hbox{and}\quad \|\rho_0-1\|_{{B}^{1}_{2,1}}+\|u_0\|_{{B}^{0}_{2,1}}\leq M,
\end{array}\end{equation}
for  some positive constants $\ud\rho, \ov\rho, M,$
  there exists a positive time $T_c$ depending only on $\ud\rho$, $\ov\rho$, $M$   such that System \eqref{system} has a unique solution
$(\rho,u,\nabla \pi)$ satisfying
\begin{equation}\label{strong solution}
\rho\in[\ud\rho,\ov\rho],\quad \|(\rho-1,u,\nabla\pi)\|_{E_{T_c}}\leq CM,\,
 \end{equation}
where $C$ is some positive constant   and  the solution space $E_T$ is the following critical nonhomogeneous Besov spaces
 $$
 E_{T}\triangleq \Bigl(C([0,T];B^1_{2,1})\cap L^1_{T}(B^3_{2,1})\Bigr)
                           \times \Bigl(C([0,T];B^0_{2,1})\cap L^1_{T}(B^2_{2,1})\Bigr)^2
                           \times \Bigl(L^1_{T}(B^0_{2,1})\Bigr)^2.
                           $$
\end{thm}

 Hence, System \eqref{system}-\eqref{cond:2d,initial} admits a unique local strong solution  $(\rho,u,\nabla\pi)$  on its lifespan $[0,T^\ast)$, $T^\ast> T_c$, with $(\rho-1,u,\nabla\pi)\in E_{t}$ for any $t<T^\ast$.
 Moreover, there exists a positive time $T_0<T^\ast$ such that  $(\rho-1,u)|_{t=T_0}\in B^3_{2,1}\times (B^2_{2,1})^2\subset H^2\times (H^1)^2$.
Just as in \cite{Danchin04}, we therefore  consider an extra pseudo-conservation law concerning $L^\infty(H^2)\times (L^\infty(H^1))^2$-norm of the  weak solutions which evolve from the initial moment $T_0$.
 This  law  ensures that the  global weak solutions  belong to the (strong) solution space $E_t$ for all   $t\geq T_0$, by virtue of the embedding $H^2\times (H^1)^2\subset B^1_{2,1}\times (B^0_{2,1})^2$, see Lemma \ref{lem:global} for more details.
Thanks to the uniqueness of the strong solutions on the time interval $[0,T^\ast)$, the Cauchy problem \eqref{system}-\eqref{cond:2d,initial}   has a unique  global strong solution $(\rho,u,\nabla\pi)$ with $(\rho-1,u,\nabla\pi)$ in $E_T$ for all $T\in (0,+\infty)$.
 By use of  Equation $\eqref{system}_1$ and the following estimates in Besov spaces (see Section \ref{sec:Besov}):
$$\|f(\rho)-f(1)\|_{B^s_{2,1}}\leq C(\ud\rho,\ov\rho) \|\rho-1\|_{B^s_{2,1}},\,
\forall f\in C^1,\, s>0,
\quad
\|gh\|_{B^a_{2,1}}\leq C\|g\|_{B^1_{2,1}} \|h\|_{B^a_{2,1}},\hbox{ if }a=0,\,1,
$$
 it is easy to find that $\d_t\rho \in L^1([0,+\infty);B^1_{2,1}(\R^2))$ and furthermore, by view of the calculation from System \eqref{system_origin} to System \eqref{system},  we have
\begin{thm}\label{thm:2d}
Let $d=2$ and Relation  \eqref{cond:coeff} hold.
For any initial data $(\rho_0,v_0)$ satisfying
 \begin{equation}\label{cond:2d,initial,v}
0<\ud\rho\leq  \rho_0\leq \ov\rho,
\quad \div(v_0+\kappa_0\nabla\ln\rho_0)=0,\quad
\hbox{and}\quad\rho_0-1\in {B}^{1}_{2,1}(\R^2),
\quad  v_0\in ({B}^{0}_{2,1}(\R^2))^2,
 \end{equation}
 System \eqref{system_origin} has a unique  global strong solution $(\rho,v,\nabla\Pi)$ with $(\rho-1,v,\nabla\Pi)$ in Space $E_T$ for any $T\in (0,+\infty)$.
\end{thm}


\begin{rem}
Let us make some remarks here:
\begin{itemize}
\item If  $(\rho,v,\nabla\Pi)$  is a solution of (\ref{system_origin}),  then
so does
\begin{equation}\label{eq:scaling1}
(\rho(\ell^2t,\ell x),\ell v(\ell^2t,\ell x),\ell^3\nabla \Pi(\ell^2t,\ell x))\quad\hbox{for all }\ \ell>0.
\end{equation}
Therefore,   the initial data $(\rho_0,v_0)$ given in \eqref{cond:2d,initial,v} is of critical regularity in the sense that  $\dot B^1_{2,1}\times \dot B^0_{2,1}$-norm is invariant
by the transform
\begin{equation}\label{eq:scaling2}
(\rho_0,v_0)(x)\rightarrow (\rho_0(\ell x),\ell v_0(\ell x))\quad\hbox{for all }\ \ell>0.
\end{equation}

\item Theorem \ref{thm:2d} implies that in dimension two, for any initial datum $\rho_0\geq \ud\rho>0$ such that $\rho_0-1\in B^1_{2,1}$, there exists a unique global-in-time solution to the quasilinear heat equation $\d_t\rho-\div(\kappa(\rho)\nabla\rho)=0$ in functional space $C_b([0,+\infty);B^1_{2,1})\cap L^1([0,+\infty);B^3_{2,1})$.

\end{itemize}
\end{rem}

\subsection{Besov spaces and some notations}\label{sec:Besov}

For completeness, let's  define Besov spaces here.
 Fix a  nonincreasing radial function $\chi$ in functional space $C^\infty_0(\R^d)$, $d\geq 2$, such that it is supported on the ball $B(0,\frac43)$ and equals to $1$ in a neighborhood of the unit ball $B(0,1)$.
   Define a sequence of functions $\{\chi_j\}_{j\geq -1}$ in light of $\chi$ as
   $$
   \chi_{-1}=\chi\quad\textrm{ and }\quad
   \chi_j(\xi)=\chi(\xi/{2^{j+1}})-\chi(\xi/{2^j})\,\textrm{ if }\, j\geq 0.
   $$
For all $j\geq -1$, denote the inverse Fourier transformation of $\chi_j$  by $h_j$.
Therefore we can define  the \emph{dyadic blocks} $(\dj)_{j\geq -1}$ as
\begin{align*}
&\dj u=\chi_j(D)u=\Int_{\R^d}h_j(y)u(x-y)\,dy \quad\hbox{  if  }\quad j\geq -1.
\end{align*}
Hence we can now define the Besov space $B^s_{p,r}$:
\begin{defi}\label{def:Besov}
For $s\in \R$, $(p,r)\in [1,\infty]^2$ and $u\in \mathcal {S}'(\R^d)$, we define
$$
B^s_{p,r}(\R^d):=\left\{u\in\mathcal{S}'(\R^d),\,\|u\|_{B^s_{p,r}}
:=\left\|\left(2^{js}\|\dj u\|_{L^p(\R^d)}\right)_{j\geq -1}\right\|_{\ell^r}<\infty\right\}.
$$
\end{defi}

We have some basic properties of Besov spaces:
 \begin{prop} \label{prop:basic}
 Let $s\in \R$, $(p,r)\in [1,\infty]^2$ and  $f$ be a continuous function from $\R$ to $\R$ such that $f(0)=0$.  Then there exists a positive constant $C$ such that
 $$\|\nabla u\|_{B^{s-1}_{p,r}(\R^d)}\leq C\|u\|_{B^{s}_{p,r}(\R^d)},$$
 and (here $C$ depends also on $\|u\|_{L^\infty}$)
 \begin{equation}\label{action}
\|f\circ u\|_{B^s_{p,r}(\R^d)}\leq C\|u\|_{B^s_{p,r}(\R^d)},\quad\textrm{if}\quad s>0.
\end{equation}
 \end{prop}

Since  we  will generally work in Sobolev spaces, we state some embedding  results (see \cite{B-C-D}):
\begin{prop}\label{prop:embed}
The following imbeddings hold true:
\begin{enumerate}[(i)]
\item $B^{s_1}_{p_1,r_1}(\R^d)\hookrightarrow B^{s_2}_{p_2,r_2}(\R^d)$ whenever $s_1-\frac{d}{p_1}> s_2-\frac{d}{p_2}$, $p_1\leq p_2$ or when $s_1=s_2$, $p_1=p_2$ and $r_1\leq r_2$.
\item The classical Sobolev space $H^s(\R^d)$ can be represented by $B^s_{2,2}(\R^d)$ and  $H^{s+1(\R^d)}\hookrightarrow B^s_{2,1}(\R^d)$, $B^s_{2,1}(\R^d)\hookrightarrow H^s(\R^d)$.
\item In dimension $d$, spaces of the form $B^{\frac dp}_{p,1}(\R^d)$ with any $p\in [1,\infty]$  can be imbedded in  $L^\infty(\R^d)$.
\end{enumerate}
\end{prop}

We recall here an apriori estimate for the following linear parabolic equation  in Besov spaces in dimension $2$ (see Proposition 4.1  in \cite{Danchin-Liao}\footnote{The a priori estimates can be extended to any dimension $d\geq 2$ and more general Besov spaces.}):
\begin{prop}\label{prop:rho}
Let $s\in (-1,1]$. Let $a(t,x)\in \cS([0,T]\times \R^2)$ satisfy
\begin{equation}\left\{\begin{array}{ccc}
\d_t a+u\cdot\nabla a-\div(\kappa\nabla a)&=&f, \\
a|_{t=0}&=&a_0,
\end{array}\right.\end{equation}
where $u=u(t,x)\in \R^2$, $\kappa=\kappa(t,x)\in \R^+$ are known and   $\kappa\geq \ud\kappa>0$. Then there exists a constant $C $ depending on $\ud\kappa$ and $s$ such that for any $t\in (0,T]$,
\begin{equation}\label{linearest:rho}
\begin{array}{c}
\|a\|_{L^\infty_t(B^{s}_{2,1}(\R^2))\cap L^1_t(B^{s+2}_{2,1}(\R^2))}
\qquad\qquad\qquad\qquad\qquad\qquad\qquad\qquad\qquad\qquad\qquad\qquad  \\
\leq \Bigl(\|a_0\|_{B^{s}_{2,1}(\R^2)}+ C \|\Delta_{-1}a\|_{L^1_t(L^{2}(\R^2))}
+\|f\|_{L^1_t(B^{s}_{2,1}(\R^2))}\Bigr)\\
\times \exp\Bigl\{C \Bigl( \|\nabla u\|_{L^1_t(B^1_{2,1})}
                 +\|\nabla \kappa\|_{L^2_t(B^{1}_{2,1}(\R^2))}^2\Bigr)\Bigr\}.
\end{array}
\end{equation}
\end{prop}

We introduce   also an estimation for products in Besov spaces  in dimension $2$, which is needed in dealing with nonlinear terms:
\begin{prop}\label{prop:product}
Let $s\in(-1,1]$. Then there exists some positive constant $C$  such that
\begin{equation}\label{product}
\|ab\|_{B^s_{2,1}(\R^2)}\leq C\|a\|_{B^1_{2,1}(\R^2)}\|b\|_{B^s_{2,1}(\R^2)}.
\end{equation}
\end{prop}

Let's give the proof for the reader's convenience. Firstly, let's focus on the paraproduct
$$T_ab:=\sum_{j\geq -1} \Bigl(\sum_{j'\leq j-2} \Delta_{j'}a \Bigr) \dj b.$$
Since the Fourier transform of $(\sum_{j'\leq j-2} \Delta_{j'}a \, \dj b)$ is supported near the annulus of size $2^j$ centered at the origin, it is easy to see that for any $s\in \R$,
$$\|T_ab\|_{B^s_{2,1}}
\leq C\sum_{j} 2^{js} \Bigl\|\sum_{j'\leq j-2} \Delta_{j'}a\Bigr\|_{L^\infty} \|\dj b\|_{L^2}
\leq C\|a\|_{L^\infty}\|b\|_{B^s_{2,1}}
\leq C\|a\|_{B^1_{2,1}} \|b\|_{B^s_{2,1}}.$$
Similarly, if $s=1$, then the above also holds for $T_b a$. Otherwise, if $s<1$, then we can calculate the paraproduct $T_b a$ as following:
$$\displaylines{
\|T_b a\|_{B^s_{2,1}(\R^2)}
\leq C\sum_{j} 2^{js}  \Bigl(\sum_{j'\leq j-2} \|\Delta_{j'}b\|_{L^\infty}\Bigr) \|\dj a\|_{L^2}
\hfill\cr\hfill
\leq C\sum_{j} \Bigl(\sum_{j'\leq j-2} 2^{(j-j')(s-1)} 2^{j'(s-1)}\|\Delta_{j'}b\|_{L^\infty}\Bigr)
 2^j\|\dj a\|_{L^2}
 \hfill\cr\hfill
\leq C\|b\|_{B^{s-1}_{\infty,1}(\R^2)} \|a\|_{B^1_{2,1}(\R^2)}
\leq C\|a\|_{B^1_{2,1}(\R^2)} \|b\|_{B^s_{2,1}(\R^2)}.
}$$
At last, we consider the remainder
$$
R(a,b):=\sum_q\Bigl(\Delta_q a\sum_{q-1\leq q'\leq q+1} \Delta_{q'}b \Bigr).
$$
Since the Fourier transform of $(\Delta_q a\sum_{q-1\leq q'\leq q+1} \Delta_{q'}b)$ is supported near the ball of size $2^q$ centered at the origin, one easily finds that if $s>-1$, then
$$\displaylines{
\|R(a,b)\|_{B^s_{2,1}(\R^2)}
\leq \|R(a,b)\|_{B^{s+1}_{1,1}(\R^2)}
\leq C\sum_{j} 2^{j(s+1)}   \sum_{q\geq j-2} \|\Delta_q a\|_{L^2}
\sum_{q-1\leq q'\leq q+1} \|\Delta_{q'}b \|_{L^2}
\hfill\cr\hfill
\leq C\sum_{j}\sum_{q\geq j-2}  2^{j(s+1)} \Bigl(2^{-q}\|a\|_{B^1_{2,1}}\Bigr)
\Bigl(2^{-qs}\|b\|_{B^s_{2,1}} \Bigr)
\leq C\|a\|_{B^1_{2,1}} \|b\|_{B^s_{2,1}}.
}$$
This completes the proof of Proposition \ref{prop:product}.


\subsection*{Notations}
 Let us fix   some notations which will be used  throughout in the sequel:
\begin{itemize}
\item We always take $\varrho=\rho-1$ in any environment, that is, $\varrho_0=\rho_0-1$, $\varrho^{\eps}=\rho^\eps-1$, etc.

\item We will always view the physical coefficients as functions of the density $\rho$ and denote for example, $\kappa'\triangleq \frac{d\kappa}{d\rho}$, $\kappa^\eps\triangleq \kappa(\rho^\eps)$, etc.

\item $C$ denotes some harmless constant which may depend on the lower bound $\ud\rho$ and the upper bound $\ov\rho$ of the initial density  (see (\ref{initial})).  In some places, we shall alternately use  the notation $A\lesssim B$ instead of $A\leq CB$.

\item Functions of the form $\langle f\rangle_\eps$ will always be viewed as the regularized functions of $f$, in the sense specified in Section \ref{sec:regular} (see (\ref{regular})).

\item If $X,Y$ are two Banach spaces, then the notation $X\hookrightarrow Y$ means that space $X$  can be imbedded into space $Y$ continuously while $X\hookrightarrow\hookrightarrow Y$ says that the embedding from $X$ to $Y$ is furthermore compact.

\item We write $u_n\rightarrow u$ in some Banach space $X$ to represent the strong convergence of the sequence $\{u_n\}_n$ to $u$ in space $X$ such that $\|u_n-u\|_X\rightarrow 0$, while $u_n\rightharpoonup u$  and $u_n\mathop{\rightharpoonup}\limits^\ast u$ in $X$ mean that $\{u_n\}_n$ converges to $u$ in the associated weak and weak-$\ast$ topology of space $X$ respectively.

\item The index $p'$ denotes the conjugate of $p$ such that $\frac{1}{p}+\frac{1}{p'}=1$.

\item If $X$ is a Banach space, $T>0$ and $p\in[1,+\infty]$, then
$L_T^p(X)$ stands for the set of Lebesgue measurable functions $f$ from
$[0,T)$ to $X$ such that $t\mapsto \|f(t)\|_{X}$ belongs to $L^p([0,T)).$
If $T=+\infty,$ then the space is merely denoted by $L^p(X).$
Finally, if $I$ is some interval of $\R$ then the notation $\cC(I;X)$ stands
for the set of continuous functions from $I$ to $X.$

\item $L^2_w(X)$ denotes the Lebesgue space $L^2(X)$  endowed with weak topology.

\item We shall keep the same notation  $X$ to designate vector-fields with components in $X.$
\end{itemize}


The rest of the paper unfolds as follows. The next section is devoted to proving Theorem \ref{thm:global} and Theorem \ref{thm:global,v} whereas the proof of Theorem \ref{thm:2d} is left in the third section.

\section{Global existence of the weak solution}
In this section we will prove the global-in-time existence of weak solutions, i.e. Theorem \ref{thm:global} and Theorem \ref{thm:global,v}.
The first    paragraph is devoted to the case when the initial density $\rho_0$ satisfies $\rho_0-1\in L^2(\R^d)$    and the second, is  to the case where $\rho_0-1\in H^1(\R^2)$ or $H^1(\R^3)$.

\subsection{The case with the density of lower regularity}
 In this subsection we will prove Theorem \ref{thm:global} in two steps. The first step, i.e. Section \ref{sec:regular}, is to solve the regularized system of the Cauchy problem \eqref{system}-\eqref{initial} while the second subsection, is devoted to show that the convergent limit of the obtained regular solution sequence is indeed a weak solution of this Cauchy problem.

\subsubsection{Regularized system}\label{sec:regular}

In this subsection we will consider the regularized system of \eqref{system}-\eqref{initial}. More precisely, let us fix a nonnegative function $\varphi\in C^\infty_0(\R^d)$ such that
$$\textrm{Supp } \varphi\in B(0,1),\quad 0\leq\varphi\leq 1,\quad \int_{\R^d}\varphi=1,$$
and consequently define a sequence of functions $\{\varphi_\eps\}_\eps$ such that $\varphi_\eps(x)=\eps^{-d}\varphi(\eps^{-1}x),\, \forall x\in \R^d$. Given any $f\in \cD'(\R^d)$, we set the regularized functions $\{\langle f\rangle_\eps\}_\eps$ as
\begin{equation}\label{regular}
\langle f\rangle_\eps\triangleq \varphi_\eps\ast f.
\end{equation}
 Now we regularize the Cauchy problem as following
 \footnote{We notice here that we don't have to regularize the coefficient $\mu$ since the density $\rho$ as a solution of the regularized equation $\eqref{system_eps}_1$ is already smooth, whereas we remain the regularized form of the coefficient $\kappa$ to keep uniform with $\eqref{system_eps}_1$. This implies the uniform enregy bound for $u$.}
\begin{equation}\label{system_eps}\left\{\begin{array}{ccc}
\d_t\rho+\div(\rho\,\langle u \rangle_\eps)-\div(\langle\kappa\rangle_\eps\nabla\rho)&=&0,\\
\d_t(\rho u)+\div\Bigl((\rho \langle u \rangle_\eps-\langle\kappa\rangle_\eps\nabla\rho)\otimes u\Bigr)-\div(2 \mu   Au)+\nabla\pi&=&0,\\
\div u&=&0,\\
\rho|_{t=0}&=&\langle\rho_0\rangle_\eps,\\
 u|_{t=0}&=&\langle u_0\rangle_\eps.
\end{array}\right.\end{equation}
It is easy to see that if initial data $(\rho_0,u_0)$ satisfies \eqref{initial}, then we have the following properties (keep in mind that $\varrho_0=\rho_0-1$):
$$\langle\varrho_0\rangle_\eps,\langle u_0\rangle_\eps\in H^\infty,\quad \langle\varrho_0\rangle_\eps\rightarrow \varrho_0\textrm{ in }L^2,\quad \langle u_0\rangle_\eps\rightarrow u_0\textrm{ in }L^2,
\quad 0<\ud\rho\leq  \langle\rho_0\rangle_\eps\leq \ov\rho,\quad \div\langle u_0\rangle_\eps=0.$$


In the following  we will use fixed point method to solve System (\ref{system_eps}). More precisely, for any $T\in (0,\infty)$ fixed, we will show that the operator $F$ from some known functions $(\tilde \rho,\tilde u)$ to the solution $(\rho,u)$ of the system (with $\tilde\kappa=\kappa(\tilde\rho)$)
\begin{equation}\label{operator}\left\{\begin{array}{ccc}
\d_t\rho+\div(\rho\,\langle\tilde u\rangle_\eps)-\div( \langle\tilde\kappa\rangle_\eps\nabla\rho)&=&0,\\
\d_t(\rho u)+\div\Bigl((\rho \langle\tilde u\rangle_\eps-\langle\tilde\kappa\rangle_\eps\nabla\rho)\otimes u\Bigr)-\div(2 \mu(\rho)  Au)+\nabla\pi&=&0,\\
\div u&=&0,\\
\rho|_{t=0}&=&\langle\rho_0\rangle_\eps,\\
 u|_{t=0}&=&\langle u_0\rangle_\eps,
\end{array}\right.\end{equation}
is compact in the Banach space
\begin{align}\label{Banach}
E_{R_0,T}=\Bigl\{(\rho,u)\,|\,&(\rho-1,u)\in C([0,T];L^2)\cap L^2(0,T;H^1),\quad \div u=0\textrm{ on }\R^d\times [0,T],\notag\\
& 0<\ud\rho\leq \rho\leq \ov \rho,\quad  \|(\rho-1,u)\|_{L^\infty(0,T;L^2)},\|(\nabla\rho,\nabla u)\|_{L^2(0,T;L^2)}\leq R_0\Bigr\},
\end{align}
with $R_0$ depending only on the initial data, to be determined later. Let us first notice that although System \eqref{operator} is nonlinear, after getting the  solution $\rho$ to  the linear system  $\eqref{operator}_1-\eqref{operator}_4$, the equations $\eqref{operator}_2-\eqref{operator}_3$ for the velocity $u$ and the pressure $\pi$ become linear immediately. Besides, thanks to the regularization, we will show  that the solution $(\rho,u)$ to System \eqref{operator} belongs to a much more regular solution space than $E_{R_0,T}$, which provides $F$ with compacity.


Now we state the well-posedness result for System \eqref{system_eps}:
\begin{prop}\label{prop:regular}
For any positive time $T\in (0,+\infty)$, there exists a unique  smooth solution $(\rho,u)\in E_{R_0,T}$ to System \eqref{system_eps} such that $(\rho-1,u)\in C([0,T];H^\infty)$.
\end{prop}

\begin{proof}Throughout the proof, we will use frequently the notation $C_\eps$ to denote the constants which may depend on $\eps,T,\ud\rho,\ov\rho$ and $R_0$ whereas the notation $C_\eps(m)$, denotes those  constants $C_\eps$ depending additionally on $m$.

We consider first the linear equation for $\varrho$:
\begin{equation}\label{density,eps,eq}\left\{\begin{array}{ccc}
\d_t \varrho+\langle\tilde u\rangle_\eps\cdot\nabla\varrho-\div(\langle\tilde\kappa\rangle_\eps\nabla\varrho)&=&0,\\
\varrho|_{t=0}&=&\langle\varrho_0\rangle_\eps.
\end{array}\right.
\end{equation}

To solve it, we will use Friedrich's method.
For any $n\in\N$, we define the space $L^2_n$ to be the closed set of $L^2$ functions with Fourier transform supported in the ball of center $0$ and radius $n$ and the associated orthogonal projector $P_n$ is defined by
$\hat{P_n f}(\xi)=1_{|\xi|\leq n}\hat f(\xi)$, then we immediately get  a unique solution $\varrho_n\in C([0,T];L^2_n)\cap C^1((0,T);L^2_n)$ to the following system:
\begin{equation}\label{density_n}\left\{\begin{array}{ccc}
\d_t\varrho_n+P_n(\langle\tilde u\rangle_\eps\cdot\nabla\varrho_n)-P_n\div(\langle\tilde\kappa\rangle_\eps\nabla\varrho_n)&=&0,\\
\varrho_n|_{t=0}&=&P_n\langle\varrho_0\rangle_\eps.
\end{array}\right.\end{equation}
In fact, it is easy to see that  the above equation  is a linear ordinary differential equation on $L^2_n$.

 Now taking the $L^2(\R^d)$ inner product between (\ref{density_n}) and $\varrho_n$ and integrating by parts give the following a priori estimate (noticing that $\div \langle\tilde u\rangle_\eps=0$, $P_n\varrho_n=\varrho_n$ and $\langle P_n f,g\rangle_{L^2}=\langle f,P_n g\rangle_{L^2}$):
$$
\frac12\frac{d}{dt}\|\varrho_n\|_{L^2}^2+\Int_{\R^d} \langle\tilde\kappa\rangle_\eps|\nabla\varrho_n|^2=0,
$$
that is,
\begin{equation}\label{density,n,L2}
\|\varrho_n\|_{L^\infty_T(L^2)}^2+C\|\nabla\varrho_n\|_{L^2_T(L^2)}^2
\leq \|\varrho_n|_{t=0}\|_{L^2}^2\leq \|\varrho_0\|_{L^2}^2.
\end{equation}
Similarly,  we can multiply (\ref{density_n}) by $\Delta \varrho_n$, $\Delta^2\varrho_n$, $\cdots$ and integrate on the whole space $\R^d$, to get
\begin{equation}\label{density,n,Hk}
\|\varrho_n\|_{L^\infty_T(H^m)}^2+\|\varrho_n\|_{L^2_T(H^{m+1})}^2
\leq C_\eps(m),\quad\forall m\geq 0.
\end{equation}
Hence, the fact that $\varrho_n$ solves (\ref{density_n}) implies that
\begin{equation}\label{density,n,dt}
 \|\d_t \varrho_n\|_{L^\infty_T(H^m)}
 \leq \|\langle\tilde u\rangle_\eps\cdot\nabla\varrho_n\|_{L^\infty_T(H^m)}
 +\|\langle\tilde\kappa\rangle_\eps\nabla\varrho_n\|_{L^\infty_T(H^{m+1})}\leq C_\eps(m).
 \end{equation}
Thus by Inequality \eqref{density,n,Hk},  Inequality \eqref{density,n,dt} and Arzel\`{a}-Ascoli Theorem, there exists one unique $\varrho\in C([0,T];H^\infty)$ such that  for any fixed $m\geq 0$, one has a convergent subsequence $\{\varrho_{n(m)}\}$ with
\begin{equation}\label{density,n,conv}
\varrho_{n(m)} \rightarrow \varrho \textrm{ in }L^\infty_T(H^m_\loc).
\end{equation}

 Now we rewrite (\ref{density_n}) as
 \begin{equation}\label{density,n,equation}
 \d_t \varrho_n+\langle\tilde u\rangle_\eps\cdot\nabla\varrho_n-\div(\langle\tilde\kappa\rangle_\eps \nabla\varrho_n)
 =(\Id-P_n)\div(\langle\tilde u\rangle_\eps \varrho_n-\langle\tilde\kappa\rangle_\eps\nabla\varrho_n).
 \end{equation}
 Since $\forall s\in\R$, we have
 \begin{equation}\label{source}
 \|(\Id-P_n)f\|_{H^s}
\leq \frac1n\|f\|_{H^{s+1}},
 \end{equation}
 thus let $n(m)$ go to $\infty$, then the above control and the convergence result (\ref{density,n,conv}) implies that the limit $\varrho\in C([0,T];H^\infty)$ really solves (\ref{density,eps,eq}) and satisfies Estimates  (\ref{density,n,L2}), (\ref{density,n,Hk}) and \eqref{density,n,dt}.
Moreover, by maximum principle, we have
\begin{equation}\label{density,bound}0<\ud \rho\leq \varrho+1\leq \ov\rho.\end{equation}

\smallbreak

Now we move to solve the following system  in $C([0,T];H^\infty)$  with  $\varrho\in C([0,T];H^\infty)$ satisfying \eqref{density,eps,eq}  given by the Step 1 (which amounts to solving $(\ref{operator})_2-\eqref{operator}_3-\eqref{operator}_5$):
\begin{equation}\label{u,eps,eq}\left\{\begin{array}{ccc}
\rho\d_t u+(\rho\langle\tilde u\rangle_\eps-\langle\tilde\kappa\rangle_\eps\nabla\rho)\cdot\nabla u-\div(2 \mu  Au)+\nabla\pi&=&0,\\
\div u&=&0,\\
 u|_{t=0}&=&\langle u_0\rangle_\eps.
\end{array}\right.\end{equation}
 We will proceed exactly as above.
  Firstly, we look for $u_n\in C([0,T];L^2_n)\cap C^1((0,T);L^2_n)$ satisfying
\begin{equation}\label{u,n,equation}\left\{\begin{array}{ccc}
\d_t u_n +P_n(Lu_n)&=&0,\\
u_n|_{t=0}&=&P_n\langle u_0\rangle_\eps,
\end{array}\right.\end{equation}
where the linear operator $L$ is defined by
\footnote{We multiply $\eqref{u,eps,eq}_1$ by $\rho^{-1}$ and rewrite the quantity $-\rho^{-1}\div(2\mu Au)$ into a summation of one 2-order  term and one 1-order term by use of $\div u=0$.}
\begin{equation}\label{L}
Lu_n=\Bigl(\langle\tilde u\rangle_\eps-\langle\tilde\kappa\rangle_\eps\rho^{-1}\nabla\rho\Bigr)\cdot\nabla u_n+\rho^{-1}\nabla\mu\cdot D u_n-\rho^{-1}\div(\mu\nabla u_n)+\rho^{-1}\nabla\pi_n,
\end{equation}
with $\nabla\pi_n$ satisfying
\begin{align}\label{pi,n,eq}
\div(\rho^{-1}\nabla\pi_n)=-\div\Bigl((\langle\tilde u\rangle_\eps-\rho^{-1}\langle\tilde\kappa\rangle_\eps\nabla\rho)\cdot\nabla u_n+2 \mu \nabla\rho^{-1}\cdot Au_n\Bigr).
\end{align}
We point out here that the following equality holds true:
\begin{equation}\label{equality}
\div(\rho^{-1}\div(2\mu Au))\equiv -\div(2\mu\nabla\rho^{-1}\cdot Au).
\end{equation}

It is easy to see   that the unique map from $u_n$ to $\nabla\pi_n$ defined by the equation (\ref{pi,n,eq})
 is continuous such that
 \begin{equation}\label{pi,n,L2}
 \|\nabla\pi_n\|_{L^2}
 \leq C_\eps\|\nabla u_n\|_{L^2}\leq C_\eps(n)\|u_n\|_{L^2_n}.
 \end{equation}
Therefore  the linear map $u_n\mapsto P_n(Lu_n)$ is continuous on $L^2_n$ which, ensures  one unique solution $u_n\in C([0,T];L^2_n)\cap C^1((0,T);L^2_n)$ of System (\ref{u,n,equation}).


Now we are at the point to get the uniform estimates for $u_n$ to show the convergence. Taking the $L^2$ inner product between (\ref{u,n,equation}) and $u_n$  implies that
$$\frac12 \frac{d}{dt}\Int_{\R^d}|u_n|^2+\Int_{\R^d}\mu\rho^{-1}  |\nabla u_n|^2
\leq C_\eps\|\nabla u_n\|_{L^2}\|u_n\|_{L^2}.$$
Thus by H\"{o}lder's inequality the following uniform estimate for $u_n$ holds:
\begin{equation}\label{u,n,L2}
\|u_n\|_{L^\infty_T(L^2)}+\|\nabla u_n\|_{L^2_T(L^2)}\leq C_\eps\|u_0\|_{L^2}.
\end{equation}
Since  by induction we have from \eqref{pi,n,eq} that
$$\|\nabla\pi_n\|_{H^m}\leq C_\eps(m)\|\nabla u_n\|_{H^m},$$
 we can multiply (\ref{u,n,equation}) by $\Delta^m u_n$, with $m\geq 0$  to derive
\begin{equation}\label{u,n,Hk}
\|u_n\|_{L^\infty_T(H^m)}+\|\d_t u_n\|_{L^\infty_T(H^m)}+\|\nabla\pi_n\|_{L^\infty_T(H^m)}\leq C_\eps(m),\quad\forall m\geq 0.
\end{equation}
Hence there exist $u\in C([0,T];H^\infty)$ and $\nabla\pi\in L^\infty((0,T];H^\infty)$ which is given by  (\ref{pi,n,eq}) with $u_n$ replaced by $u$, such that $u(0)=\langle u_0\rangle_\eps$ and for any fixed $m\geq 0$, there exists a subsequence $u_{n(m)}$, $\nabla\pi_{n(m)}$  verifying
$$u_{n(m)} \rightarrow u \textrm{ in }L^\infty_T(H^m_\loc),\quad \nabla\pi_{n(m)}\rightharpoonup \nabla\pi\textrm{ in }L^\infty_T(H^m_\loc).$$
Moreover, \eqref{source} entails
\begin{equation}\label{u,eqn}
\d_t u+ Lu=0.
\end{equation}
Applying the divergence operator $\div$ to it yields
\footnote{By Definition \eqref{L} of the operator $L$, Equation \eqref{pi,n,eq} of $\pi$ and Equality \eqref{equality}, we have
$$\div Lu=\div(\rho^{-1}\nabla\mu\cdot Du-\rho^{-1}\div(\mu\nabla u)+\rho^{-1}\div(2\mu Au))
=\div(\rho^{-1}\nabla\mu\cdot Du-\rho^{-1}\div(\mu Du))
=-\div(\mu\rho^{-1}\nabla\div u).$$}
\begin{equation}\label{div u,eq}\d_t(\div u)-\div\Bigl(\mu\rho^{-1}\nabla\div u\Bigr)=0.\end{equation}
 Thus the parabolic equation (\ref{div u,eq}) ensures that $\div u=0$. Therefore $u$ truly solves (\ref{u,eps,eq}).
  Furthermore, we take the inner product between (\ref{u,eps,eq}) and $u$, issuing the Energy Equality (\ref{energy iden:velocity}) by use of Equation (\ref{density,eps,eq}) for $\rho$, which together with the identity $\|Au\|_{L^2}=\|\nabla u\|_{L^2}$ (by $\div u=0$) entails
 \begin{equation}\label{u,L2}
 \|u\|_{L^\infty([0,t];L^2)}+\|\nabla u\|_{L^2((0,t];L^2)}\leq C(\ud\rho,\ov\rho)\|u_0\|_{L^2},\quad\,\forall t>0.
 \end{equation}

\smallbreak

At last, noticing (\ref{density,n,L2}) and (\ref{u,L2}), we just have to choose $R_0$ depending on $\|\varrho_0\|_{L^2},\|u_0\|_{L^2},\ud\rho,\ov\rho$ such that the operator $F:(\tilde\rho,\tilde u)\mapsto (\rho,u)$  maps from $E_{R_0,T}$ to $E_{R_0,T}$. Furthermore, the boundedness \eqref{density,n,Hk}, \eqref{density,n,dt} and \eqref{u,n,Hk} ensures that
$$
F:E_{R_0,T}\mapsto E_{R_0,T}\cap\Bigl\{(\rho,u)\,|\, \|(\varrho,\d_t\rho,u,\d_t u)\|_{L^\infty_T(H^m)}\leq C_\eps(m),\,\forall m\geq 0\Bigr\},
$$
which implies that $F$ is compact in $C([0,T];H^m_\loc)$ for all $m\geq 0$. Hence there exists one unique fixed point $(\varrho,u)\in C([0,T]; H^\infty)$ of the operator $F$ in $E_{R_0,T}$ which is also the solution to System (\ref{system_eps}). This ends the proof of Proposition \ref{prop:regular}.
\end{proof}

\begin{rem}\label{rem:linear,global}
Since in Banach space $E_{R_0,T}$, the bound $R_0$ is independent of the time $T$,  Proposition \ref{prop:regular} actually permits that for any $\eps$, there exists a unique globally-in-time existing smooth solution $(\varrho^\eps,u^\eps)\in C([0,+\infty);H^\infty)$ to System \eqref{system_eps} such that $\forall \eps>0$,
\begin{equation}\label{bound_eps}
\rho^\eps\in [\ud\rho,\ov\rho],\, \|\varrho^\eps\|_{L^\infty(L^2)}+\|\nabla\rho^\eps\|_{L^2(L^2)}\leq C\|\varrho_0\|_{L^2(\R^d)},\, \|u^\eps\|_{L^\infty(L^2)}+\|\nabla u^\eps\|_{L^2(L^2)}\leq C\|u_0\|_{L^2(\R^d)},
\end{equation}
with   $C$ being a constant depending only on $\ud\rho,$ $\ov\rho$.
\end{rem}


\subsubsection{The convergence to a weak solution}\label{sec:convergence}
We now are at the point to  show that  the solution sequence given by Section \ref{sec:regular} converges to a weak solution to Cauchy problem \eqref{system}-\eqref{initial}.
The smoothing effect on both variables is  useful to use the compactness methods and the strategy is hence quite standard.
So let's just sketch the proof.

  By Remark \ref{rem:linear,global} at the end of Section \ref{sec:regular},  we  may assume that there exist subsequences $\{\rho^{\eps_n}\}_n$ and $\{u^{\eps_n}\}_n$ of the solution sequences $\{\rho^\eps\}_\eps$ and $\{u^\eps\}_\eps$ respectively such that
$$
\varrho^{\eps_n}\mathop{\rightharpoonup}\limits^\ast \varrho\textrm{ in }L^\infty(L^2\cap L^\infty), \quad u^{\eps_n}\mathop{\rightharpoonup}\limits^\ast u\textrm{ in }L^\infty(L^2),\quad \nabla\varrho^{\eps_n}\rightharpoonup \nabla\varrho,\quad \nabla u^{\eps_n}\rightharpoonup\nabla u\textrm{ in }L^2(L^2),
$$
with the limit $(\rho,u)$ verifying  Estimate \eqref{bound_eps} too.

Thanks to the uniform bound on $\nabla\rho$ and by use of the regularized system $\eqref{system_eps}_1-\eqref{system_eps}_4$, we can easily use   a diagonal process to show that $\rho$ solves Equation $\eqref{system}_1$   in distribution sense.
For example, there exists a subsequence, still denoted by $\kappa^{\eps_n}$, such that $\langle\kappa^{\eps_{n }}\rangle_{\eps_n}-\kappa \rightarrow 0, a.e.$
Thus $\langle\kappa^{\eps_{n }}\rangle_{\eps_n} \nabla\rho^{\eps_n} \rightharpoonup \kappa\nabla\rho $ in $L^2(L^2(\R^d))$.

The above bound  \eqref{bound_eps} furthermore ensures that  Equation $\eqref{system}_1$   holds in $L^2_\loc(H^{-1})$.
Thus, we can test it by the solution $\varrho\in L^2(H^1)$ itself such that
Energy Equality \eqref{energy iden:density} hold for $\varrho^\eps$ and $\varrho$ both (notice that $\div u=0$ in $L^2_w([0,\infty)\times \R^d)$), i.e.  for all $t\in[0,\infty)$,
\begin{equation}\label{ee,density_eps}
\frac 12 \|\varrho^\eps(t)\|_{L^2}^2+\|\langle\kappa^\eps\rangle_\eps^{\frac 12} \nabla\varrho^\eps\|_{L^2_t(L^2)}^2=\frac 12 \|\langle\varrho_0\rangle_\eps\|_{L^2}^2,\quad
\frac 12 \|\varrho(t)\|_{L^2}^2+\|\kappa^{\frac 12} \nabla\varrho\|_{L^2_t(L^2)}^2=\frac 12 \|\varrho_0\|_{L^2}^2.
\end{equation}
Now we consider the quantity
$$\frac 12\|\varrho^{\eps_n}(t)-\varrho(t)\|_{L^2(\R^d)}^2
      +\|\langle\kappa^{\eps_n}\rangle_{\eps_n}^{\frac 12} \nabla\varrho^{\eps_n}-\kappa^{\frac 12} \nabla\varrho\|_{L^2_t(L^2)}^2,$$
 which by Energy Equality \eqref{ee,density_eps}, is equal to
 $$
 \frac 12 \|\langle\varrho_0\rangle_{\eps_n}\|_{L^2(\R^d)}^2+\frac 12 \|\varrho_0\|_{L^2(\R^d)}^2-\Bigl\langle \varrho^{\eps_n}(t),\varrho(t)\Bigr\rangle_{L^2(\R^d)}
-2\Bigl\langle\langle\kappa^{\eps_n}\rangle_{\eps_n}^{\frac 12} \nabla\varrho^{\eps_n},\kappa^{\frac 12} \nabla\varrho\Bigr\rangle_{L^2_t(L^2)}.
$$
Since   we have also $ \langle\kappa^{\eps_{n}}\rangle_{\eps_{n}}^{\frac 12}\nabla\varrho^{\eps_{n}}\rightharpoonup \kappa^{\frac 12} \nabla\varrho$  in
$L^2(L^2)$, the above quantity converges to
$$\|\varrho_0\|_{L^2}^2-\|\varrho(t)\|_{L^2}^2-2\|\kappa^{\frac 12} \nabla\varrho\|_{L^2_t(L^2)}^2=0.$$
 This implies  that
$$\varrho^{\eps_n}\rightarrow \varrho\textrm{ in }L^\infty(L^2) \quad\textrm{and}\quad
\langle\kappa^{\eps_n}\rangle_{\eps_n}^{\frac 12} \nabla\varrho^{\eps_n}\rightarrow \kappa^{\frac 12} \nabla\varrho\textrm{ in }
L^2(L^2).$$
Thus, by the boundedness of $\|\varrho^{\eps_n}\|_{L^\infty([0,\infty)\times \R^d)}$, we have
$\varrho^{\eps_n}\rightarrow \varrho\textrm{ in }L^\infty(L^p)$, $\forall p\in [2,\infty)$. Therefore, $\varrho|_{t=0}=\varrho_0$ in $L^p$ for all $p\in [2,\infty)$.

\smallbreak 

The following statement concerning the velocity $u$ follows exactly Proof  of Theorem 2.4 in the P.-L. Lions's book \cite{PLions96}.
Let us recall it briefly for the reader's convenience.
Let us first observe that the Sobolev embedding ensures that $\{u^\eps\}_\eps$ is bounded in $L^\infty_T(L^2)\cap L^2_T(L^{\frac{2d}{d-2}})$ (or $L^2_T(L^p)$ with $p\in [2,\infty)$ if $d=2$) for any positive finite time $T$ and hence
\begin{align*}
&\{\langle u^\eps\rangle_\eps\otimes u^\eps\} \textrm{ is bounded in } L^\infty_T(L^1)\cap L^2_T(L^{\frac{d}{d-1}})\,(\textrm{or }L^2_T(L^p)\textrm{ with } p\in [1,2)\textrm{ if } d=2), \\
\textrm{ and }& \{\nabla\rho^\eps\otimes u^\eps\}\textrm{ is bounded in } L^2_T(L^1)\cap L^1_T(L^{\frac{d}{d-1}})\,(\textrm{or }L^1_T(L^p)\textrm{ with } p\in [1,2)\textrm{ if } d=2).
\end{align*}
Therefore, in view of the equation $(\ref{system_eps})_2$ for $u^\eps$, there exist constants $p\in (2,\infty)$, $m>1$
 \footnote{We notice that $H^{m_1}\hookrightarrow W^{m_2,q}$ if $m_1-\frac d2\geq m_2-\frac dq,\, q\geq 2$.} 
and $M$ depending on $T,R_0$ such that  for all divergence-free function $\phi\in L^p_T(H^m)$ we have
\begin{equation}\label{velocity:dt,bound}
\left|\Bigl\langle \d_t (\rho^\eps u^\eps),\phi\Bigr\rangle_{\cD',\cD}\right|
=\left|\Bigl\langle-\Bigl(\rho^\eps\langle u^\eps\rangle_\eps-\langle\kappa^\eps\rangle_\eps\nabla\rho^\eps \Bigr)\otimes u^\eps+ 2 \mu^\eps  Au^\eps,\nabla\phi\Bigr\rangle_{\cD',\cD}\right|
\leq M\|\phi\|_{L^p_T(H^m)}.
\end{equation}
Let us notice that the Leray projector $\cP:=\Id+\nabla(-\Delta)^{-1}\div$ is bounded on each Sobolev space $H^s$. Hence from (\ref{velocity:dt,bound}), we actually have $\d_t(\cP(\rho^\eps u^\eps))$ is bounded in $L^{p'}_T(H^{-m})$. Since $\rho^{\eps_n} u^{\eps_n}$ converges weakly to $\rho u$, the boundedness of $\{\rho^\eps u^\eps\}$ in $L^\infty_T(L^2)$ implies  the existence of one convergent subsequence  (still denoted by $\rho^{\eps_n}u^{\eps_n}$) such that $\cP(\rho^{\eps_n}u^{\eps_n})\rightarrow \cP(\rho u)$ in $C([0,T];L^2_w)$.
Hence, we have for any $t>0$,
\begin{align*}
\Int^t_0\Int_{\R^d}\rho^{\eps_n} |u^{\eps_n}|^2
=\Int^t_0 \langle \cP(\rho^{\eps_n}u^{\eps_n}), u^{\eps_n}\rangle
\rightarrow \Int^t_0 \langle \cP(\rho u ), u \rangle
=\Int^t_0\Int_{\R^d}\rho  |u |^2.
\end{align*}
Thus $\rho^{\eps_n}u^{\eps_n}\rightarrow \rho u $ and $ u^{\eps_n}\rightarrow u$ in $L^2_\loc(L^2)$.
Moreover, we apply Theorem 2.2 in \cite{PLions96} to get $u\in C([0,\infty);L^2_w)$.
It is easy to find that
$$\rho^{\eps_n}\langle u^{\eps_n}\rangle_{\eps_n}\otimes u^{\eps_n}\rightarrow \rho u\otimes u\textrm{ in }L^2_T(L^1),
\quad \mu^{\eps_n}Au^{\eps_n}\rightharpoonup \mu Au\textrm{ in }L^2_T(L^2),$$
and $
\langle\kappa^{\eps_n}\rangle_{\eps_n}\nabla\rho^{\eps_n}\otimes u^{\eps_n}\rightarrow \kappa \nabla\rho\otimes u
\textrm{ in }
L^1_T(L^1)$.
Thus observing Equation $(\ref{system_eps})_2$, there exists some distribution of gradient form $\nabla\pi$ such that Equation $(\ref{system})_2$ holds for the above limit $u$ at least in distribution sense and hence \eqref{weak solution:u} holds.
According to Theorem 2.2 in \cite{PLions96}, the conclusion $u\in C(L^2_w)$ results from the initial assumption $\div u_0=0$.

This completes the proof of Theorem \ref{thm:global}.


\subsection{The case with the density of higher regularity}\label{sec:H1}

In this  section we will tackle the case with smoother  density.
 It is easy to see that proving Theorem \ref{thm:global,v} equals to proving the following:\\
\,\, Let $d=2, 3$. For any initial data $(\rho_0,u_0)$ such that
\begin{equation}\label{initial,H1}
0<\ud\rho\leq \rho_0\leq \ov\rho,\quad \rho_0-1\in H^1(\R^d),
\quad u_0\in L^2(\R^d),\quad \div u_0=0,
\end{equation}
System \eqref{system} has a weak solution $(\rho,u)$ satisfying \eqref{sol:H1} and \eqref{rho,v,energy}.

Thus in the first paragraph of this section, by establishing a new a priori estimate in smoother functional space in dimension $2$, we deduce that if \eqref{initial,H1} holds for the initial data, then so do the weak solutions $(\rho,u)$ got in the last section.
 However in dimension $3$, we will reprove the existence of weak solutions by regularizing the system in two levels which, ensures us to get  the uniform estimate \eqref{EE:rho-H1,3D} when the transport velocity is still regularized.
 This will be done in the second paragraph.

\subsubsection{$2$D case}

We will establish two lemmas (Lemma \ref{lem:pi} and Lemma \ref{lem:est,L2}), in order to show that the global weak solutions $(\varrho,u,\nabla\pi)=(\rho-1,u,\nabla\pi)$ given by Theorem \ref{thm:global} but with smoother initial density $\varrho_0\in H^1(\R^2)$,  are bounded only by initial data in the Banach space $X_2(T)\times X_1(T)\times X_{-1}(T)$ for all $T\in [0,\infty]$ in dimension $2$ with
\begin{align*}
&X_2(T)\triangleq \{\varrho\in L^\infty_T(H^1(\R^2))\,|\, \nabla\varrho\in L^2_T(H^1(\R^2)),\quad 0<\ud\rho\leq \varrho(t)+1\leq \ov\rho,\,\forall\, t\in [0,T]\},\\
&X_1(T)\triangleq \{u\in L^\infty_T(L^2(\R^2))\, |\, \nabla u\in L^2_T(L^2(\R^2))\},\\
&X_{-1}(T)\triangleq (X_1(T))':\textrm{ the dual of the Banach space }X_1(T).
\end{align*}
However, in order  to prove uniqueness and stability, it is not enough to just consider the solutions in
$X_2(T)\times X_1(T)\times X_{-1}(T)$, since $H^1(\R^2)$ can not be embedded into $L^\infty(\R^2)$,
which is needed in estimating the nonlinear terms.
Therefore we will consider the initial data in the critical Besov spaces, in order to get a unique global strong solution.
This will be done in next section.

Throughout this section, we will use frequently (explicitly or implicitly) the following Gagliardo-Nirenberg inequality in dimension $2$:
\begin{equation}\label{ineq:GN}
\| f\|_{L^4(\R^2)}\lesssim \|f\|_{L^2(\R^2)}^{1/2}\|\nabla f\|_{L^2(\R^2)}^{1/2}.
\end{equation}

We notice that by (\ref{ineq:GN}), if  $\varrho\in X_2(T)$, then the mapping $h\mapsto f(\rho)h$ is an isomorphism on $X_1(T)$ (and hence $X_{-1}(T)$) for any diffeomorphism $f$ from $[\ud\rho,\ov\rho]$ to  $\R$. In fact, we have
$$\|f(\rho)h\|_{X_1(T)}\leq
C(f,\|\rho\|_{L^\infty_T(L^\infty)})
(\|h\|_{X_1(T)}+\|\nabla\rho\|_{L^4_T(L^4)}\|h\|_{L^4_T(L^4)})
\leq C(f,\|\varrho\|_{X_2(T)})\|h\|_{X_1(T)}.$$


We now prove first that $\nabla\pi\in X_{-1}(T)$ if $\varrho\in X_2(T),u\in X_1(T)$. Indeed, we just have to show that the convergent regular sequence $\nabla\pi^{\eps_n}$ which are solutions of the following equation (see Equation (\ref{pi,n,eq}))
\begin{equation}\label{pi,eq}
\div(\rho^{-1}\nabla\pi)
=-\div((\langle u\rangle_\eps-\rho^{-1}\langle\kappa\rangle_\eps\nabla\rho)\cdot\nabla u
          +2\mu\nabla\rho^{-1}\cdot Au),
\end{equation}
 are uniformly bounded in $X_{-1}(T)$. Hence we introduce the following lemma:
\begin{lem}\label{lem:pi}
Assume that the smooth triplet $(\varrho,u,\nabla\pi)$ satisfies Equation \eqref{pi,eq},
then there exists one constant $C_1$ depending only on $\|\varrho\|_{X_2(T)}$, $\|u\|_{X_1(T)}$ such that
\begin{equation}\label{pi,H-1}
\|\nabla\pi\|_{X_{-1}(T)}\leq C_1.
\end{equation}
\end{lem}

\begin{proof}
 The proof is very similar to  the proof of Lemma 2.1 in \cite{PLions96}. Given any $h\in X_1(T)$, we first claim that we have the  decomposition $h=h_1+h_2$ such that
 \begin{equation}\label{est:decomposition}
\cR\wedge(\rho h_1)=0,\quad \div h_2=0,\quad
\|h_i\|_{X_1(T)}\leq C(\|\varrho\|_{X_2(T)})\|h\|_{X_1(T)},\quad i=1,2,
\end{equation}
where  $\cR_i=(-\Delta)^{-1/2}\frac{\d}{\d x_i}$, $i=1,2$ denotes the usual Riesz transform and $f\wedge g:= f_1g_2-f_2g_1$.

In fact, we decompose the function $\rho h\in X_1(T)$ such that
$$\rho h=\tilde h_1+\tilde h_2=-\cR(\cR\cdot (\rho h))+(\Id+\cR(\cR\cdot))(\rho h),\quad \cR\wedge \tilde h_1=0,\quad \cR\cdot\tilde h_2=0,$$
and
$$\|\tilde h_i\|_{X_1(T)}\leq C\|\rho h\|_{X_1(T)}\leq C(\|\varrho\|_{X_2(T)})\|h\|_{X_1(T)},\, i=1,2.$$
To prove (\ref{est:decomposition}) then amounts to searching for a unique function $\cR v\in X_1(T)$ such that
$$h_1=\rho^{-1}\tilde h_1-\rho^{-1}\cR v,\quad h_2=\rho^{-1}\tilde h_2+\rho^{-1}\cR v,$$
with
\begin{equation}\label{lemma,v,eq}\div(\rho^{-1}\nabla V+\rho^{-1}\tilde h_2)=0,\quad V=(-\Delta)^{-1/2}v.\end{equation}
According to Section 3 in \cite{Danchin10euler}, Equation (\ref{lemma,v,eq}) admits one unique solution $\nabla V=\cR v$ such that
$$\|\cR v\|_{L^\infty_T(L^2)}\leq C\|\tilde h_2\|_{L^\infty_T(L^2)}\leq C\|\rho h\|_{L^\infty_T(L^2)}
\leq C(\|\varrho\|_{X_2(T)})\|h\|_{L^\infty_T(L^2)}.$$
Now we take the derivative $\nabla$ to (\ref{lemma,v,eq}) to arrive at
$$\div(\rho^{-1}\nabla^2 V+\nabla V\otimes \nabla\rho^{-1}+(\nabla(\rho^{-1}\tilde h_2))^T)=0,$$
which similarly gives
$$\|\nabla\cR v\|_{L^2_T(L^2)}\leq C(\|\varrho\|_{X_2(T)})\Bigl(\|\nabla V\|_{L^4_T(L^4)}\|\nabla\rho^{-1}\|_{L^4_T(L^4)}+\|\rho^{-1}\tilde h_2\|_{X_1(T)}\Bigr).$$
Thus by (\ref{ineq:GN}) and Young's inequality (\ref{est:decomposition}) follows.

\medskip

By the decomposition (\ref{est:decomposition}), we have for any $h\in X_1(T)$,
 $$\langle\nabla\pi,h\rangle_{X_{-1}(T),X_1(T)}=\langle\nabla\pi, h_1\rangle=\langle\rho^{-1}\nabla\pi, \rho h_1\rangle.$$
Therefore Equation (\ref{pi,eq}) for $\nabla\pi$ yields
\begin{align*}
\left|\langle\nabla\pi,h\rangle\right|
&=\left|\Bigl\langle (\langle u\rangle_\eps-\rho^{-1}\langle\kappa\rangle_\eps\nabla\rho)\cdot\nabla u+2\mu\nabla\rho^{-1}\cdot Au, \rho h_1\Bigr\rangle\right|\\
&\leq C(\|u\|_{L^4_T(L^4)}+\|\nabla\rho\|_{L^4_T(L^4)})\|\nabla u\|_{L^2_T(L^2)}\|\rho h_1\|_{X_1(T)}\\
&\leq C(\|\varrho\|_{X_2(T)},\|u\|_{X_1(T)})\|h_1\|_{X_1(T)},
\end{align*}
which gives the lemma by \eqref{est:decomposition}.
\end{proof}

\begin{rem}\label{rem:X}
We point out here that, $L^2_T(H^{-1}(\R^2))\subset X_{-1}(T)$  by definition and $\Delta u\in L^2_T(H^{-1}(\R^2))$ since $\nabla u\in L^2_T(L^2(\R^2))$. But since the ``low'' regularity $\varrho\in X_2(T)$ can not permit that $h\mapsto \rho h$ is an isomorphism on $L^2_T(H^1(\R^2))$ due to a lack of control on $\|\nabla\rho\otimes h \|_{L^2_T(L^2)}$, it is not clear that $\|\nabla\pi\|_{L^2_T(H^{-1}(\R^2))}$ is bounded in Lemma \ref{lem:pi}. Even if $\rho\equiv 1$, that is in the classical incompressible Navier-Stokes equation case, it is well-known that the pressure term $\nabla\pi$ can be  bounded in $L^2_T(\cM(\R^2))$ with $\cM(\R^2)$ the measure space, denoting the dual space of $C_0(\R^2)$.
\end{rem}

It is easy to check that the following lemma concerning the unknowns $\rho,u,\nabla\pi$ immediately  yields Theorem \ref{thm:global,v} when $d=2$, which is omitted here.

\begin{lem}\label{lem:est,L2}
Let $d=2$. For any initial data $(\rho_0,u_0)$ satisfying \eqref{initial,H1},   the global weak solution $(\varrho,u,\nabla\pi)$ given by Theorem \ref{thm:global} satisfies
\begin{equation}\label{2destimate}
\|\varrho\|_{X_2(T)}+\|u\|_{X_1(T)}+\|\nabla\pi\|_{X_{-1}(T)}+\|\d_t\rho\|_{L^2_T(L^2)}+\|\d_t u\|_{X_{-1}(T)}\leq C_2,
\end{equation}
with $C_2$ depending only on $\ud\rho,\ov\rho,\|\varrho_0\|_{H^1(\R^2)},\|u_0\|_{L^2(\R^2)}$.
Moreover, Energy Equality \eqref{energy iden:velocity} holds and $\varrho\in C([0,\infty);H^1(\R^2)$, $u\in C([0,\infty);L^2(\R^2))$.
\end{lem}


\begin{proof} By Theorem \ref{thm:global}, Equation $\eqref{system}_1$ and Lemma \ref{lem:pi}, to prove (\ref{2destimate}) it rests to prove $\varrho\in X_2(T)$ and $\d_t u\in X_{-1}(T)$, which can be assumed to be right  a priori. In fact, since if $u,\nabla\rho\in X_1(T)$, then $\rho u,\kappa\nabla\rho\in X_1(T)$ by \eqref{ineq:GN} and hence
\begin{equation}\label{est:rho_t}
\|\d_t\rho\|_{L^2_T(L^2(\R^2))}=\|-\div(\rho u-\kappa\nabla\rho)\|_{L^2_T(L^2(\R^2))}\leq C(\|\varrho\|_{X_2(T)},\|u\|_{X_1(T)}).
\end{equation}

Let us deal with $\varrho$ first. Set $K$ to be an antiderivative of $\kappa$ such that $K(1)=0$. Since  $\nabla K=\kappa\nabla\rho\in X_1(T)$, then $K\in L^\infty_T(H^1\cap L^\infty)$ and $\nabla K\in L^2_T(H^1)$. Multiplying $(\ref{system})_1$ by $\kappa=\kappa(\rho)$ yields
\begin{equation}\label{eq:K}
\d_t K+u\cdot\nabla K-\kappa\Delta K=0\textrm{ in }L^2_T(L^2).
\end{equation}
Taking the $L^2(\R^2)$ inner product with $\Delta K$ issues
\begin{equation}\label{equation:K}
\frac12\frac{d}{dt}\Int_{\R^2}|\nabla K|^2-\Int_{\R^2}u\cdot\nabla K\Delta K+\Int_{\R^2}\kappa|\Delta K|^2=0.
\end{equation}
For any $\varepsilon>0$, we have from Inequality \eqref{ineq:GN}  and Young's Inequality
\begin{align*}
\Bigl| \Int_{\R^2} u\cdot\nabla K\Delta K \Bigr|\leq \|u\|_{L^4}\|\nabla K\|_{L^4}\|\Delta K\|_{L^2}
\leq \eps\|\Delta K\|_{L^2}^2+C_\eps\|u\|_{L^4}^4\|\nabla K\|_{L^2}^2,
\end{align*}
for some constant $C_\eps$.
Let us choose sufficiently small $\varepsilon$, then we have shown (by (\ref{u,energy}))
\begin{equation}\label{estimate:K}
\sup_{0\leq t\leq T}\|\nabla K\|_{L^2}^2+\int^T_0\|\Delta K\|_{L^2}^2\leq e^{C\int^T_0\|u\|_{L^4}^4}C\|\nabla\varrho_0\|_{L^2}^2\leq  e^{C\|u_0\|_{L^2}^4}C\|\nabla\varrho_0\|_{L^2}^2 ,\quad\forall\, T\geq 0.
\end{equation}
It is also easy to see from \eqref{equation:K} that $\nabla K\in C([0,\infty);L^2)$.
Thus $\varrho\in C([0,\infty);H^1)$.

 Now since $\Delta K=\kappa\Delta\varrho+\nabla\kappa\cdot\nabla\varrho$, we have
\begin{align*}
\|\Delta\varrho\|_{L^2}\leq C(\|\Delta K\|_{L^2}+\|\nabla K\|_{L^4}^2)\leq C\|\Delta K\|_{L^2}(1+\|\nabla K\|_{L^2}),
\end{align*}
which already gives the estimate for $\varrho$:
\begin{equation}\label{estimate:temp'}
\sup_{0\leq t\leq T}\|\nabla\varrho(t)\|_{L^2}^2+\int^T_0\|\nabla^2\varrho\|_{L^2}^2\leq e^{C\|u_0\|_{L^2}^4}C\|\nabla\varrho_0\|_{L^2}^2(1+\|\nabla\varrho_0\|_{L^2}^2) ,\,\forall\, T\geq 0.
\end{equation}

Now we turn to the equation for $u$. It is easy to find that
$$\d_t u=-\rho^{-1}\Bigl(\d_t\rho\, u+\div((\rho u-\kappa\nabla\rho)\otimes u)-\div(2\mu Au)+\nabla\pi\Bigr).$$
Thus for any $h\in X_1(T)$, we have by \eqref{est:rho_t}
\begin{align*}
|\langle \d_t u,h\rangle_{X_{-1}(T),X_1(T)}|&\leq
\|\d_t\rho\|_{L^2_T(L^2)}\|u\|_{L^4_T(L^4)}\|\rho^{-1}h\|_{L^4_T(L^4)}
+|\langle\nabla\pi,\rho^{-1}h\rangle|\\
&\quad +\|(\rho u-\kappa\nabla\rho)\otimes u-2\mu Au\|_{L^2_T(L^2)}\|\nabla(\rho^{-1}h)\|_{L^2_T(L^2)}\\
&\leq C(\|\varrho\|_{X_2(T)},\|u\|_{X_1(T)},\|\nabla\pi\|_{X_{-1}(T)})\|h\|_{X_1(T)}.
\end{align*}
Hence (\ref{2destimate}) follows and $(\ref{system})_2$ holds in $X_{-1}(T)$.

In order to show the Energy Equality \eqref{energy iden:velocity}, we take the $\langle\cdot,\cdot\rangle_{X_{-1}(T),X_1(T)}$ inner product between Equation $(\ref{system})_2$ and $u$ to arrive at (notice $(\ref{system})_1$ holding in $L^2_T(L^2)$)
\begin{align*}
\frac12\frac{d}{dt}\Int_{\R^2}\rho|u|^2+2\Int_{\R^2}\mu|Au|^2=0,
\end{align*}
which gives \eqref{energy iden:velocity} immediately.
This implies $u\in C([0,\infty);L^2)$.
\end{proof}

\begin{rem}\label{rem:K}
Here we have to consider the function $K$ instead of directly the density $\rho$, in order to get estimates on $\|\nabla\rho\|_{L^\infty_T(L^2)\cap L^2_T(H^1)}$. In fact, if we directly take the derivative on $\eqref{system}_1$, then whether the quantity $\nabla\rho\Delta\rho$ issuing from the divergence term $\div(\kappa\nabla\rho)$ can be killed by the ``good'' term $\Delta\nabla\rho$ is not clear.
\end{rem}

\subsubsection{$3$D case}

Unlike the last paragraph, we do not have Inequality \eqref{ineq:GN} in dimension $3$.
Thus $u\not\in L^4(L^4(\R^3))$ and hence   the quantity $\int_{\R^3}u\cdot\nabla K\Delta K$ in \eqref{equation:K} doesn't make sense.
However, inspired by the computations before Theorem \ref{thm:global,v}, we can show first the uniform bounds as \eqref{EE:rho-H1,3D} for the regular approximated solutions $\rho^\eps$ and then, by lower semi-continuity of the $L^2$-norm, it holds for the weak solution $\rho$.

More precisely, for any $\eps>0$ and any $\delta>0$, we will consider the following regularized system of Cauchy problem \eqref{system}-\eqref{initial,H1}, instead of System \eqref{system_eps}\footnote{Due to Remark \ref{rem:K} and $u\not\in L^4(L^4)$, we regularize the system in two levels in order to get the   bound for the $H^1$-norm of  the density $\rho$ by considering the equation for  the scalar function $K=K(\rho)$ with $\nabla K=\kappa\nabla\rho$ where the transport velocity is regularized.}:
\begin{equation}\label{system_eps_delta}\left\{\begin{array}{ccc}
\d_t\rho+\div(\rho\,\langle u \rangle_\eps)-\div( \langle \kappa \rangle_\delta \nabla\rho)&=&0,\\
\d_t(\rho u)+\div\Bigl((\rho \langle u \rangle_\eps- \langle \kappa \rangle_\delta \nabla\rho)\otimes u\Bigr)-\div(2 \mu   Au)+\nabla\pi&=&0,\\
\div u&=&0,\\
(\rho,u)|_{t=0}&=&(\langle\rho_0\rangle_\eps, \langle u_0\rangle_\eps).
\end{array}\right.\end{equation}
By view of Subsection \ref{sec:regular}, the above system has a unique solution $(\rho^{\eps,\delta},u^{\eps,\delta})$ with $\rho^{\eps,\delta}-1, u^{\eps,\delta}\in C([0,\infty);H^\infty)$, such that Estimate \eqref{bound_eps} holds uniformly in $\eps$ and $\delta$.

Following Subsection \ref{sec:convergence}, it is easy to find that there exists a global-in-time weak solution $(\rho^\eps,u^\eps)$   to the following system
\begin{equation}\label{system_eps:v}\left\{\begin{array}{ccc}
\d_t\rho+\div(\rho\,\langle u \rangle_\eps)-\div(  \kappa  \nabla\rho)&=&0,\\
\d_t(\rho u)+\div\Bigl((\rho \langle u \rangle_\eps-   \kappa   \nabla\rho)\otimes u\Bigr)-\div(2 \mu   Au)+\nabla\pi&=&0,\\
\div u&=&0,\\
(\rho,u)|_{t=0}&=&(\langle\rho_0\rangle_\eps, \langle u_0\rangle_\eps).
\end{array}\right.\end{equation}

Moreover,  thanks to the smooth transport velocity, by a similar method as in the proof of Lemma \ref{lem:est,L2}, the equation for the density $\rho^\eps$ holds in the following sense:
$$
\d_t\rho + \langle u^\eps\rangle_\eps\cdot\nabla\rho -\div(  \kappa   \nabla\rho )=0\quad\hbox{ in }\quad L^2(L^2).
$$
In fact, we suppose a priori $\rho^\eps-1\in L^\infty_T(H^1\cap L^\infty)\cap L^2_T(H^2)$ for any $T\in (0,\infty)$.
By use of the interpolation inequality in dimension $3$
\begin{equation}\label{ineq:3d}
\|\nabla\rho\|_{L^4}\lesssim \|\Delta\rho\|_{L^2}^{1/2} \|\rho\|_{L^\infty}^{1/2},
\end{equation}
one has $\nabla\rho^\eps\in L^4_T(L^4)$.
Hence the scalar function $K^\eps:=K(\rho^\eps)\in L^\infty_T(H^1\cap L^\infty)\cap L^2_T(H^2)$ with the function $K$ as defined in the proof of Lemma \ref{lem:est,L2} satisfies Equation \eqref{eq:K} with the transport velocity $\langle u^\eps\rangle_\eps$.
Taking the $L^2$-inner product between it and $\Delta K^\eps$,  applying the following inequality (noticing $\div(\langle u^\eps\rangle_\eps)=0$):
$$
\left|\Int_{\R^3} \langle u^\eps\rangle_\eps \cdot\nabla K^\eps\Delta K^\eps \right|
=\left|\Int_{\R^3} \nabla K^\eps\cdot \nabla\langle u^\eps\rangle_\eps \cdot\nabla K^\eps  \right|
\leq \|\nabla u^\eps\|_{L^2}\|\nabla K^\eps\|_{L^4}^2
\leq C \|\nabla u^\eps\|_{L^2}\|\Delta K^\eps\|_{L^2},
$$
and then performing Young's Inequality and Estimate \eqref{bound_eps}, we arrive at
$$
\|\nabla K^\eps\|_{L^\infty_T(L^2)}+\|\Delta K^\eps\|_{L^2_T(L^2)}
\leq C(\ud\rho,\ov\rho) \Bigl(\|\nabla K^\eps(0)\|_{L^2}+\|u_0\|_{L^2}\Bigr)
\leq C\Bigl(\|\nabla\rho_0\|_{L^2}+\|u_0\|_{L^2}\Bigr).
$$
One easily finds that the above estimate also holds for $\rho^\eps$.
Therefore, taking into account also Energy Identity \eqref{energy iden:density}, we arrive at for all $T\in (0,\infty)$, $\eps>0$,
\begin{equation}\label{bound_eps:H1}
\|\rho^\eps-1\|_{L^\infty_T(H^1(\R^3))}+\|\nabla\rho^\eps\|_{L^2_T(H^1(\R^3))}
\leq C\|(\rho_0-1,u_0)\|_{H^1(\R^3)\times L^2(\R^3)}.
\end{equation}

Now we let $\eps\rightarrow 0$, then the same argument as in Subsection \ref{sec:convergence} ensures that there exists a global-in-time weak solution $(\rho,u)$ to System \eqref{system} such that Estimate \eqref{rho,v,energy} holds and $\nabla\rho\in C([0,\infty);L^2_w)$.
Define $v=u-\kappa\nabla \ln\rho$, then one easily checks that Equality \eqref{weak solution:v} holds for some divergence-free test function $\phi$ by virtue of \eqref{weak solution:u}.
This completes the proof of Theorem \ref{thm:global,v}.


\section{Well-posedness in dimension two}

In this section we aim to prove Theorem \ref{thm:2d}.
By the arguments before it, we just have to prove the global-in-time existence of a unique strong solution $(\rho,u)$  to Cauchy problem \eqref{system}-\eqref{cond:2d,initial}.
Indeed, by Theorem \ref{prop:solution-besov}, it rests to show a pseudo-conservation law concerning $L^\infty(H^2)\times L^\infty(H^1)$-norm of its  weak solutions  and moreover,  such weak solutions are also strong, by use of Proposition \ref{prop:rho}.

 The following lemma supplies the needed conservation law:
\begin{lem}\label{lem:2dest,H1}
We  assume  $(\varrho,u,\nabla\pi)$ to be a weak solution to System \eqref{system} with the initial data $\varrho_0,u_0$ satisfying the following condition:
\begin{equation}\label{cond:2dinitial,H1}
0<\ud\rho\leq \varrho_0+1\leq \ov\rho,\quad \varrho_0\in H^2(\R^2),\quad u_0\in H^1(\R^2),\quad \div u_0=0,
\end{equation}
then there exists one  constant $C_3$ depending only on $\ud\rho,\ov\rho,\|\varrho_0\|_{H^2(\R^2)},\|u_0\|_{H^1(\R^2)}$ such that the following a priori estimate holds true:
\begin{equation}\label{2destimate,H1}
\sup_{0\leq t\leq T}(\|\varrho\|_{H^2}^2+\|u\|_{H^1}^2)
+\Int^T_0 (\|\nabla\varrho\|_{H^2}^2+\|\nabla u\|_{H^1}^2+\|\d_t \varrho\|_{H^1}^2+\|\d_t u\|_{L^2}^2+\|\nabla\pi\|_{L^2}^2)\leq C_3.
\end{equation}
\end{lem}

\begin{proof}
It is easy to see that (\ref{2destimate}) already holds by Lemma \ref{lem:est,L2}.
 As usual, we can assume a priori that $\varrho\in L^\infty(H^2)$, $\nabla\varrho\in L^2(H^2)$ and $u\in L^\infty(H^1)$, $\nabla u\in L^2(H^1)$.
 In the following we will use thoroughly the Gagliardo-Nirenberg inequality \eqref{ineq:GN} and the following interpolation inequality
 \begin{equation}\label{ineq:GNinfty}
\|f\|_{L^\infty(\R^2)}\lesssim \|f\|_{L^2(\R^2)}^{1/2}\|\Delta f\|_{L^2(\R^2)}^{1/2}.
\end{equation}

We first consider the equation for $u$. We take the $L^2$ inner product between Equation  $(\ref{system})_2$  and $\d_t u$ to arrive at
\begin{align}\label{equation:u'}
0&=\Int_{\R^2} \rho|\d_t u|^2+\frac{d}{dt}\Int_{\R^2} 2\mu |Au|^2 +\Int_{\R^2} \rho u\cdot\nabla u\cdot \d_t u-\kappa\nabla\rho\cdot\nabla u\cdot\d_t u-2\mu'|Au|^2\d_t\rho.
\end{align}
On the other hand, from the equation $(\ref{system})_2-\eqref{system}_3$ we have
\begin{align}\label{eq:u''}
\Delta u=\mu^{-1}\Bigl(\rho\d_t u+\rho u\cdot\nabla u-\kappa\nabla\rho\cdot\nabla u-2\mu'\nabla\rho\cdot Au+\nabla\pi\Bigr),
\end{align}
and the elliptic equation for $\pi$ (similar to \eqref{pi,n,eq})
\begin{equation}\label{eq:pi}
\div(\rho^{-1}\nabla\pi)=-\div\Bigl(( u-\rho^{-1}\kappa\nabla\rho)\cdot\nabla u+2 \mu \nabla\rho^{-1}\cdot Au\Bigr).
\end{equation}
Equation \eqref{eq:pi} above gives us the estimate for $\nabla\pi$:
\begin{align}\label{est:pi}
\|\nabla\pi\|_{L^2}
\leq C(\|u\|_{L^4}+\|\nabla \rho\|_{L^4})\|\nabla u\|_{L^4}.
\end{align}
Therefore Equality \eqref{eq:u''} entails
\begin{equation*}
\|\Delta u\|_{L^2}
\leq C\Bigl(\|\d_t u\|_{L^2}+(\|u\|_{L^4}+\|\nabla \rho\|_{L^4})\|\nabla u\|_{L^4}\Bigr),
\end{equation*}
which implies, by applying (\ref{ineq:GN}) on $\nabla u$ and Young's Inequality,
\begin{equation}\label{est:u''}
\|\Delta u\|_{L^2}\leq C\Bigl(\|\d_t u\|_{L^2}+(\|u\|_{L^4}^2+\|\nabla \rho\|_{L^4}^2)\|\nabla u\|_{L^2}\Bigr),
\end{equation}
and hence
\begin{equation}\label{est:u',L4}
\|\nabla u\|_{L^4}^2\leq C\Bigl(\|\d_t u\|_{L^2}\|\nabla u\|_{L^2}+(\|u\|_{L^4}^2+\|\nabla \rho\|_{L^4}^2)\|\nabla u\|_{L^2}^2\Bigr).
\end{equation}

Set
$$I\triangleq \Int_{\R^2} \rho u\cdot\nabla u\cdot \d_t u-\kappa\nabla\rho\cdot\nabla u\cdot\d_t u-2\mu'|Au|^2\d_t\rho,$$
then
\begin{align*}
\|I\|_{L^1([0,T])}
&\leq C \Int^T_0 (\|u\|_{L^4}+\|\nabla\rho\|_{L^4})\|\nabla u\|_{L^4}\|\d_t u\|_{L^2}
                       +\|Au\|_{L^4}^2\|\d_t\rho\|_{L^2}.
\end{align*}
Thus by Estimate \eqref{est:u',L4} and Young's Inequality, we finally arrive at
\begin{align*}
\|I\|_{L^1([0,T])}
&\leq \eps\Int^T_0 \|\d_t u\|_{L^2}^2
+C_\eps \Int^T_0 (\|u\|_{L^4}^4+\|\nabla\rho\|_{L^4}^4+\|\d_t\rho\|_{L^2}^2)\|\nabla u\|_{L^2}^2.
\end{align*}

Therefore by  the equality $\|\nabla u\|_{L^2}^2=\int |Au|^2$ and Estimate \eqref{2destimate}, we can choose sufficiently small $\varepsilon$ to deduce from \eqref{equation:u'} that
\begin{equation}\label{estimate:u'}
\sup_{0\leq t\leq T}\|\nabla u\|_{L^2}^2
+\Int^T_0 \|\d_t u\|_{L^2}^2\leq Ce^{C_2}\|\nabla u_0\|_{L^2}^2.
\end{equation}
Moreover, by estimates \eqref{est:pi}, \eqref{est:u''}, \eqref{est:u',L4} and \eqref{estimate:u'}, we derive
\begin{align}\label{estimate:u''}
\Int^T_0 \|\Delta u\|_{L^2}^2+\|\nabla u\|_{L^4}^4+\|\nabla\pi\|_{L^2}^2
\leq C(\|\varrho_0\|_{H^2},\|u_0\|_{H^1}).
\end{align}

\smallbreak

Now we turn  to the density $\rho$. We further apply ``$\Delta$'' to Equation (\ref{eq:K}) of $K$, yielding the equation for the scalar function $\cK=\Delta K\in L^\infty(L^2)\cap L^2(H^1)$:
\begin{equation}\label{equation:cK}
\d_t \cK+2\nabla u:\nabla^2 K+u\cdot\nabla\cK+\Delta u\cdot\nabla K-\Delta(\kappa \cK)=0.
\end{equation}
Hence again taking  the $L^2$ inner product between (\ref{equation:cK}) and $\cK$ shows
\begin{equation}\label{eq:cK}
\frac 12\frac{d}{dt}\Int_{\R^2} \cK^2+\Int_{\R^2} \kappa |\nabla \cK|^2+\Int_{\R^2}2(\nabla u:\nabla^2 K)\,\cK+(\Delta u\cdot\nabla K)\cK+ \kappa'\cK\nabla\rho\cdot\nabla\cK=0.
\end{equation}
Set
$$J\triangleq \Int_{\R^2}2(\nabla u:\nabla^2 K)\,\cK+(\Delta u\cdot\nabla K)\cK+ \kappa'\cK\nabla\rho\cdot\nabla\cK,$$
then noticing  $\|\nabla\rho\|_{L^4}\leq C\|\nabla K\|_{L^4}$, we have
$$\|J\|_{L^1([0,T])}\leq C\Int^T_0 (\|\nabla u\|_{L^2}\|\nabla^2 K\|_{L^4}+\|\Delta u\|_{L^2}\|\nabla K\|_{L^4}+\|\nabla K\|_{L^4}\|\nabla\cK\|_{L^2})\|\cK\|_{L^4}.$$
By use of $\|\nabla^2 K\|_{L^4},\|\cK\|_{L^4}\lesssim \|\cK\|_{L^2}^{1/2}\|\nabla\cK\|_{L^2}^{1/2},$ we have from above that
$$\|J\|_{L^1([0,T])}\leq \eps\Int^T_0 \|\nabla\cK\|_{L^2}^2 +C_\eps\Int^T_0 (\|\nabla u\|_{L^2}^{2}+\|\nabla K\|_{L^4}^4)\|\cK\|_{L^2}^2+\|\Delta u\|_{L^2}^2 .$$
Therefore by \eqref{2destimate}, (\ref{estimate:K}),  (\ref{estimate:u''}) and Gronwall's Inequality, Inequality \eqref{eq:cK} gives
\begin{align}
\sup_{0\leq t\leq T}\|\cK(t)\|_{L^2}^2+\Int^T_0\|\nabla\cK\|_{L^2}^2
&\leq C\exp\Bigl\{\int^T_0 \|\nabla u\|_{L^2}^2+\|\nabla K\|_{L^4}^4\Bigr\}
      \left(\|\cK(0)\|_{L^2}^2+\Int^T_0 \|\Delta u\|_{L^2}^2\right) \notag\\
&\leq C(\|\varrho_0\|_{H^2},\|u_0\|_{H^1}),\label{estimate:cK}
\end{align}
and hence
\begin{equation}\label{estimate:K,L4}
\sup_{0\leq t\leq T}\|\nabla K(t)\|_{L^4}\leq \sup_{0\leq t\leq T}\|\nabla K(t)\|_{L^2}^{1/2}\|\cK(t)\|_{L^2}^{1/2}\leq C(\|\varrho_0\|_{H^2},\|u_0\|_{H^1}).
\end{equation}
Furthermore, by view of  the two identities
$$
\kappa\Delta\varrho=\cK-\nabla\kappa\cdot\nabla\varrho
\hbox{ and }
\kappa\nabla\Delta\varrho=\nabla\cK-\nabla\kappa\Delta\varrho
                                       -\nabla\varrho\cdot\nabla^2\kappa
                                       -\nabla\kappa\cdot\nabla^2\varrho,
$$
we get from  \eqref{ineq:GN} and \eqref{ineq:GNinfty}  that
\begin{equation}\label{estimate:temp''}
\sup_{0\leq t\leq T}\|\Delta \varrho(t)\|_{L^2}^2+\Int^T_0\|\nabla\Delta\varrho\|_{L^2}^2\leq C(\|\varrho_0\|_{H^2},\|u_0\|_{H^1}).
\end{equation}
Moreover, Equation $\eqref{system}_1$ ensures
$$\nabla(\d_t\rho)=-\nabla\rho\cdot Du-u\cdot\nabla^2\rho+\nabla\cK,$$
which by above yields
\begin{equation}\label{estimate:rho_t}
\Int^T_0\|\nabla\d_t\varrho\|_{L^2}^2\leq C(\|\varrho_0\|_{H^2},\|u_0\|_{H^1}),
\end{equation}
which together with the estimates (\ref{2destimate}), (\ref{estimate:u'}), (\ref{estimate:u''}) and (\ref{estimate:temp''}) give (\ref{2destimate,H1}).
\end{proof}

\begin{rem}\label{rem:pi:L1}

It is easy to deduce from \eqref{eq:pi} that
$$
-\rho^{-1} \Delta\pi=\nabla\rho^{-1}\cdot\nabla\pi+(\nabla u-\kappa\rho^{-1}\nabla\rho):\nabla u-
  \mu\nabla\rho^{-1}\cdot\Delta u.
$$
Thus by the uniform bound \eqref{2destimate,H1}, one gets
$\|\Delta\pi\|_{L^1(L^2)}\leq C(\|\varrho_0\|_{H^2(\R^2)}, \|u_0\|_{H^1(\R^2)})$.
\end{rem}

The next lemma is devoted to the weak-strong uniqueness result under the initial condition \eqref{cond:2dinitial,H1}.
\begin{lem}\label{lem:global}
Under the same hypotheses of Lemma \ref{lem:2dest,H1}, the weak solutions belong to (strong) solution space $E_T$ for any $T\in (0,+\infty)$.
\end{lem}
\begin{proof}
Denote by $C_0$ some harmless constant depending only on   $\|\varrho_0\|_{H^2(\R^2)}, \|u_0\|_{H^1(\R^2)}$ in the following.
Let's  rewrite  the equations for $\rho$ and $u$ in System \eqref{system} as following:
  \begin{equation}\label{system_para}\left\{\begin{array}{c}
  \d_t\rho-\div(\kappa\nabla\rho)=-u\cdot\nabla\rho,\\
  \d_t u-\div(\mu\rho^{-1}\nabla u)=-(u-(\kappa\rho^{-1}+\mu\rho^{-2})\nabla\rho)\cdot\nabla u-\rho^{-1}\nabla\mu\cdot Du-\rho^{-1}\nabla\pi.
  \end{array}\right.\end{equation}

By Proposition \ref{prop:basic} and Proposition \ref{prop:product}, the $L^1_T(B^1_{2,1})$ (resp. $L^1_T(B^0_{2,1})$)-norm of the right hand side of Equation $\eqref{system_para}_1$ (resp. $\eqref{system_para}$)  can be bounded by the  following two quantities respectively:
$$
C\Int^T_0  \|  u \|_{B^1_{2,1}} \|\nabla\rho\|_{B^1_{2,1}} \,dt
  \quad\hbox{and}\quad
C\Int^T_0 \Bigl((\|  u \|_{B^1_{2,1}}+\|\varrho\|_{B^2_{2,1}}) \|\nabla u\|_{B^0_{2,1}}
  + \|\varrho\|_{B^1_{2,1}} \|\nabla\pi\|_{B^0_{2,1}}\Bigr)\,dt.
$$
Since  $\|\varrho\|_{B^1_{2,1}}\leq \|\varrho\|_{H^2}$ and $\|\nabla\pi\|_{B^0_{2,1}}\lesssim \|\nabla\pi\|_{L^2}+\|\Delta\pi\|_{L^2}$, we have from Estimate \eqref{2destimate,H1} and Equation \eqref{eq:pi} that
$$\displaylines{
\Int^T_0  \|\varrho\|_{B^1_{2,1}} \|\nabla\pi\|_{B^0_{2,1}} \,dt
\lesssim
  \Int^T_0  C_0\Bigl(\|\Delta\pi\|_{L^2} +
  \|(u-\kappa\rho^{-1}\nabla\rho)\cdot\nabla u+2\mu\nabla\rho^{-1}\cdot Au\|_{L^2} \Bigr)\,dt.
}$$
Because  $B^0_{2,1}\hookrightarrow L^2$, the product estimates entails
$$
\Int^T_0  \|\varrho\|_{B^1_{2,1}} \|\nabla\pi\|_{B^0_{2,1}} \,dt
\leq
  C_0 \Int^T_0  \|\Delta\pi\|_{L^2}
  +  C_0\Int^T_0  (\|  u \|_{B^1_{2,1}}+\|\varrho\|_{B^2_{2,1}}) \|\nabla u\|_{B^0_{2,1}} \,dt.
$$

On the other side, it is easy to check that for any $C^1(\R,\R)$-function $f$ with $f(1)=0$,
$$\|f(\rho)\|_{L^\infty(H^2)}+\|\nabla f(\rho)\|_{L^2(H^2)}\leq C_0.$$

Therefore, for any solution $(\varrho,u)$ to System \eqref{system}, Proposition \ref{prop:rho} tells us that
$$\displaylines{
\|\varrho\|_{L^\infty_T(B^1_{2,1})\cap L^1_T(B^3_{2,1})}
+\|u\|_{L^\infty_T(B^0_{2,1})\cap L^1_T(B^2_{2,1})}
\hfill\cr\hfill
\leq C_0\Bigl(1 + \Int^T_0 \|(\varrho,u)\|_{L^2}
+ \Int^T_0 \|  u \|_{B^1_{2,1}} \|\nabla\rho\|_{B^1_{2,1}}
+ (\|  u \|_{B^1_{2,1}}+\|\varrho\|_{B^2_{2,1}}) \|\nabla u\|_{B^0_{2,1}}
+  \|\Delta\pi\|_{L^2}  \Bigr).
}$$
As we have interpolation inequalities
$$
\| u\|_{B^1_{2,1}}\lesssim \|u\|_{B^0_{2,1}}^{1/2}\|u\|_{B^2_{2,1}}^{1/2},
\quad
\| \varrho\|_{B^2_{2,1}}\lesssim \|\varrho\|_{B^1_{2,1}}^{1/2}\|\varrho\|_{B^3_{2,1}}^{1/2},$$
hence, by Young's Inequality and Gronwall's Inequality again, the above inequality implies
$$\displaylines{
\|\varrho\|_{L^\infty_T(B^1_{2,1})\cap L^1_T(B^3_{2,1})}
+\|u\|_{L^\infty_T(B^0_{2,1})\cap L^1_T(B^2_{2,1})}
\hfill\cr\hfill
\leq C_0\exp\Bigl\{\Int^T_0   \|\nabla\rho\|_{B^1_{2,1}}^2+  \|\nabla u\|_{B^0_{2,1}}  ^2\Bigr\}
\Bigl(1 + \|(\varrho,u)\|_{L^1_T(L^2)}
+  \|\Delta\pi\|_{L^1_T(L^2)}  \Bigr).
}$$
It is easy to find from Estimate \eqref{2destimate,H1} and Remark \ref{rem:pi:L1} that
$$
\|\varrho\|_{L^\infty_T(B^1_{2,1})\cap L^1_T(B^3_{2,1})}
+\|u\|_{L^\infty_T(B^0_{2,1})\cap L^1_T(B^2_{2,1})}
\leq C_0\
\Bigl(1 + \|(\varrho,u)\|_{L^1_T(L^2)} \Bigr).
$$
This completes the proof.
\end{proof}

Now with  Lemma \ref{lem:global} in hand, we will check that there exists a globally existing strong solution to Cauchy problem \eqref{system}-\eqref{cond:2d,initial}.
 Firstly,  Theorem \ref{prop:solution-besov} ensures that there exists a unique strong solution $(\rho_1,u_1,\nabla\pi_1)$ on its lifespan $[0,T^\ast)$ for some positive time $T^\ast> T_c$.
Hence there exists   $T_0\in (0,T^\ast)$ such that $\varrho_1(T_0)\in B^2_{2,1}\hookrightarrow H^2$ and $u_1(T_0)\in B^1_{2,1}\hookrightarrow H^1$.
Let $(\rho_2,u_2,\nabla\pi_2) $ to be a weak solution which evolves from the initial data   $\rho_1(T_0),u_1(T_0)$.
Thus, since Lemma \ref{lem:global} ensures that $(\rho_2,u_2,\nabla\pi_2)\in E_T$ for any positive finite $T$, it coincides with the strong solution $(\rho_1,u_1,\nabla\pi_1)$  on the time interval $[T_0,T^\ast)$ by the uniqueness result.
 Furthermore, we can choose $T^\ast=+\infty$. In fact, if $T^\ast<+\infty$, then by the global-in-time boundedness of $\|\rho_2-1\|_{L^\infty(B^1_{2,1})}$ and $\|u_2\|_{L^\infty(B^0_{2,1})}$ given by \eqref{2destimate,H1}, the lifespan  of $(\varrho_1,u_1,\nabla\pi_1)$ should go beyond $[0,T^\ast)$.
 This is a controdiction.

\subsection*{Acknowledgment}
The author would like to thank Professor Didier Bresch sincerely for informing the global-in-time existence result for System \eqref{system_origin} in \emph{smooth bounded domain}, under exactly the same coefficient relationship assumption \eqref{cond:coeff}.
 The author would also like to take the opportunity to thank her supervisor Professor Rapha\"{e}l Danchin  for fruitful discussions.

\bibliographystyle{plain}
\bibliography{myreference}

\end{document}